\documentclass[11pt,twoside,reqno]{amsart}

\usepackage[margin=1in]{geometry} 
\usepackage{amsmath,amsthm,amssymb,bm,mathabx}
\usepackage{color} 
\usepackage{mathtools,mathptmx}
\usepackage{indentfirst}

\usepackage{amsbsy} 
\usepackage{amsfonts}
\usepackage{graphicx}
\usepackage{tkz-euclide}

\usepackage{array}
\usepackage{geometry}
\usepackage{cases}

\usepackage{hyperref}
\hypersetup{colorlinks=true, allcolors=blue}

\usepackage[square, numbers, comma, sort&compress]{natbib} 
\usepackage[all]{xy}
\usepackage{blkarray}
\usepackage{datetime}
\usepackage{enumitem} 
\usepackage{relsize}
\usepackage[super]{nth}
\usepackage{imakeidx}
\usepackage[normalem]{ulem}
\usepackage{cancel}

\newtheorem{theorem}{Theorem}[section]
\newtheorem{proposition}[theorem]{Proposition}
\newtheorem{corollary}[theorem]{Corollary}
\newtheorem{lemma}[theorem]{Lemma}
\newtheorem{remark}[theorem]{Remark}

\theoremstyle{definition}

\numberwithin{equation}{section}
\numberwithin{theorem}{section}

\newcommand{\dis}{\displaystyle}
\newcommand{\R}{\mathbb{R}}
\newcommand{\C}{\mathbb{C}}
\newcommand{\Z}{\mathbb{Z}}
\newcommand{\N}{\mathbb{N}}

\newcommand{\pochhammer}[2][n]{\left(#2\right)_{#1}}

\newcommand{\Hypergeometric}[5][x]{{}_{#2} F_{#3} \left( \left.\begin{matrix} #4 \\ #5 \end{matrix} \, \right| \, #1\right)}

\newcommand{\floor}[1]{\left\lfloor #1 \right\rfloor}
\newcommand{\ceil}[1]{\left\lceil #1 \right\rceil}
\newcommand{\seq}[2][n\in\N]{\left(#2\right)_{#1}}

\newcommand{\e}{\mathrm{e}}

\newcommand{\Functional}[2]{#1\left[#2\right]}

\newcommand{\StieltjesRogersPoly}[3][r]{S_{#2}^{\,(#1)}\left(#3\right)}
\newcommand{\modifiedStieltjesRogersPoly}[4][r]{S_{#2}^{\,(#1;\,#3)}\left(#4\right)}
\newcommand{\generalisedStieltjesRogersPolyTypeJ}[5][r]{S_{#2,#3}^{\,(#1;\,#4)}\left(#5\right)}

\setlist[itemize]{leftmargin=*,topsep=6pt,partopsep=3pt,itemsep=6pt}
\setlist[enumerate]{leftmargin=*,topsep=6pt,partopsep=3pt,itemsep=6pt,label=(\roman*)}

\title[Bidiagonal matrix factorisations related to multiple orthogonal polynomials]
{Bidiagonal matrix factorisations related \\ to multiple orthogonal polynomials}
		
\author[Branquinho]{Am\'ilcar Branquinho$^{1}$}
\address{$^1$CMUC, Departamento de Matemática,
 Universidade de Coimbra, 3001-454 Coimbra, Portugal}
\email{$^1$ajplb@mat.uc.pt}

\author[Díaz]{Juan E.F. D\'iaz$^{2}$}
\address{$^{2,3,4}$CIDMA, Departamento de Matemática, Universidade de Aveiro, 3810-193 Aveiro, Portugal}
\email{$^2$juan.enri@ua.pt}

\author[Foulquié]{Ana Foulqui\'e-Moreno$^{3}$}
\email{$^3$foulquie@ua.pt}

\author[Lima]{H\'elder Lima$^{4}$}
\email{$^4$helder.lima@ua.pt}

\author[Mañas]{Manuel Ma\~nas$^{5}$}
\address{$^5$Departamento de Física Teórica, Universidad Complutense de Madrid, Plaza Ciencias 1, 28040-Madrid, Spain 
}
\email{$^5$manuel.manas@ucm.es}



\begin{document}

\begin{abstract}
We provide necessary and sufficient conditions for the Hessenberg recurrence matrix associated with a system of multiple orthogonal polynomials to admit a factorisation as a product of bidiagonal matrices. 
Using the Gauss-Borel factorisation of the moment matrix, we show that the nontrivial entries of those bidiagonal matrices can be expressed in terms of coefficients of type I or type II multiple orthogonal polynomials on the step-line with respect to the original system and its Christoffel transformations. 
Using the connection of multiple orthogonal polynomials with branched continued fractions, we show that the nontrivial entries of the bidiagonal matrices in the factorisation of the Hessenberg recurrence matrix correspond to the coefficients of a branched continued fraction associated with the given system of multiple orthogonal polynomials.
As a case study, we present an explicit bidiagonal factorisation for the Hessenberg recurrence matrices of the Jacobi-Pi\~neiro polynomials and, as a limiting case, the multiple Laguerre polynomials of first kind.
\end{abstract}

\maketitle

\tableofcontents

\section{Introduction and motivation}
The focal point of this paper are factorisations 
\begin{align}
\label{bidiagonal matrix factorisation L_1...L_r U introduction}
\mathrm{H}=\mathrm{L}_1\cdots\mathrm{L}_r\,\mathrm{U},
\end{align}
of banded Hessenberg recurrence matrices of systems of multiple orthogonal polynomials as products of $r$ lower-bidiagonal matrices $\mathrm{L}_1,\cdots,\mathrm{L}_r$, and one upper-bidiagonal matrix $\mathrm{U}$ of the form
\begin{align}
\label{bidiagonal matrices U and L_k}
\mathrm{L}_i=
\begin{bmatrix} 
1\\
\alpha_i & 1 \\
& \alpha_{r+1+i} & 1 \\
& & \ddots & \ddots
\end{bmatrix},
\; 1\leq i\leq r,
\quad \text{and} \quad
\mathrm{U}=
\begin{bmatrix} 
\alpha_0 & 1 \\
& \alpha_{r+1} & 1 \\
& & \alpha_{2(r+1)} & 1 \\ 
& & & \ddots & \ddots
\end{bmatrix},
\end{align}
whose nontrivial entries $\seq[m\in\N]{\alpha_m}$ are all different from zero.

When these entries are all positive, the Hessenberg matrix  $\mathrm{H}=\mathrm{L}_1\cdots\mathrm{L}_r\,\mathrm{U}$ is oscillatory. 
Spectral properties of oscillatory matrices lead to a Favard-type theorem for Hessenberg matrices that admit a positive bidiagonal factorisation, in terms of a system of multiple orthogonal polynomials with respect to positive measures on $\R^+$. For these systems, the zeros of the type II polynomials on the step-line are all simple, real, and positive, and the zeros of consecutive polynomials interlace (see \cite[Thm.~6]{AnaAmilcarManuel-Osc.BandedHessenberg...} and \cite[Thm.~4.8]{MOPassociatedwithBCFforRatiosOfHypergeometricSeries}).

The study of these positive bidiagonal factorisations dates back to \cite{GeneticSumsAptekarevEtAl}, where high-order nonsymmetric difference operators are considered. The bidiagonal factorisation arises in connection with the vector Stieltjes continued fraction and serves as a tool for analysing both the direct and inverse spectral problems associated with the operator. It is also applied to the integration of the discrete KdV equation hierarchy. 

The Hessenberg matrix arising from the recurrence relations of multiple orthogonal polynomials can be interpreted as the transition probability matrix of a Markov chain that extends beyond the classical birth-and-death framework. 
These positive bidiagonal factorisations correspond to Markov chains whose stochastic matrices can be decomposed into products of bidiagonal stochastic matrices, each representing pure birth-and-death Markov chains (see \cite{Jacobi-PineiroMarkovChains,AnaAmilcarManuel-Osc.BandedHessenberg...}). 
An example of Markov chains with a physical motivation are urn models.
Some urn models have been solved in terms of classical orthogonal polynomials and, more recently, a new urn model based on the bidiagonal factorisation of the Jacobi Pi\~neiro multiple orthogonal polynomials with respect to two measures was proposed in \cite{UrnModelJacobi-Pineiro}.

In the literature we find several examples of bidiagonal factorisations for the Hessenberg recurrence matrices of multiple orthogonal polynomials with respect to systems of $2$ measures.
In particular, in \cite[\S 8]{GeneticSumsAptekarevEtAl}, we can find this bidiagonal factorisation for the Jacobi-Pi\~neiro polynomials, and, in \cite{BidiagFactHahnMOP}, we can find it for the Hahn multiple orthogonal polynomials and their descendants in the multiple Askey scheme, including the Jacobi-Pi\~neiro polynomials, multiple Meixner (of the first and second kind), multiple Laguerre (of the first and second kind), multiple Kravchuk, and multiple Charlier. 
Nevertheless, the bidiagonal factorisation of the Hahn multiple orthogonal polynomials was obtained only after substantial computational effort and required an independent proof, as no general algorithm was available. 
Furthermore, in \cite{AnaAmilcarManuel-Bidiag.Fact.andDarboux} is obtained a bidiagonal factorisation in terms of the values of the type I and II multiple orthogonal polynomials evaluated at the origin, corresponding to Darboux transformations of the Hessenberg matrix, for the Jacobi-Pi\~neiro and hypergeometric multiple orthogonal polynomials, again with respect to $2$ measures.

Bidiagonal factorisations of the Hessenberg recurrence matrices associated with multiple orthogonal polynomials with respect to systems of any number of measures whose moment generating functions are hypergeometric were obtained in \cite{MOPassociatedwithBCFforRatiosOfHypergeometricSeries}, using a connection of those polynomials with branched continued fractions.
That paper provides an explicit example of a new family of multiple orthogonal polynomials, for which the location, simplicity, and interlacing of the zeros were proved using the bidiagonal factorisation of the corresponding Hessenberg recurrence matrix and there is no other known method to prove these properties without imposing significantly stronger conditions on the parameters involved.

The main goals of this paper are to develop methods to obtain bidiagonal factorisations of the form \eqref{bidiagonal matrix factorisation L_1...L_r U introduction} for the Hessenberg recurrence matrices of systems of multiple orthogonal polynomials and to establish necessary and sufficient conditions for the existence of those factorisations.
%
In Section \ref{Factorisations using MOP coefficients}, we give a sufficient condition for the existence of the bidiagonal factorisation. 
In this approach, the nontrivial entries $\seq[n\in\N]{\alpha_n}$ in the bidiagonal matrices are described in terms of the coefficients of the type I or type II polynomials along the step-line, for both the original system and its Christoffel transformations. 
In Section \ref{Factorisations linked to BCF}, we obtain the same sequence $\seq[n\in\N]{\alpha_n}$  using the connection between multiple orthogonal polynomials, production matrices and branched continued fractions as introduced in \cite{AlanSokal.MOP-ProdMat-BCF}. This framework allows us to establish necessary and sufficient conditions for the existence of the bidiagonal factorisation.
Finally, in Section \ref{Jacobi-Pineiro section}, we present a case study: we construct the explicit bidiagonal factorisation of the Hessenberg recurrence matrix of the Jacobi-Pi\~neiro polynomials with an arbitrary number of weights. As a limiting case, we also derive the corresponding factorisation for the multiple Laguerre polynomials of first kind.

\section{Background}

\subsection{Multiple orthogonal polynomials}
\hspace*{\fill}

Multiple orthogonal polynomials (see \cite[\S~4]{NikishinSorokinBook}, \cite[\S~23]{IsmailBook}, and \cite[\S 3]{WalterSurveyAIMS2018}) are a generalisation of orthogonal polynomials satisfying orthogonality conditions with respect to systems of several measures rather than just one measure.
In this section, we briefly revisit the theory of multiple orthogonal polynomials, defining the orthogonality conditions via linear functionals in $\C[x]$, the ring of polynomials in one variable with complex coefficients.
For more generality, $\C$ could be replaced by a generic field in all definitions and theoretical results. 
That generality is not relevant for this work, but it may be of interest to some readers.

A linear functional $u$ defined on $\C[x]$ is a linear map $u:\C[x]\to\C$.
The action of a linear functional $u $ on a polynomial $f\in\C[x]$ is denoted by $\Functional{u}{f}$.
The moment of order $n\in\N$ of a linear functional $u $ is equal to $\Functional{u}{x^n}$.
By linearity, every linear functional $u $ is uniquely determined by its moments. 
In most relevant examples, there exists a measure $\mu$ supported on a subset of the real line or on a curve in the complex plane such that $\dis\Functional{u}{f}=\int f(x)\mathrm{d}\mu(x)$ for any $f\in\C[x]$ and we say that the functional $u$ is induced by the measure $\mu$.
However, the use of functionals allows us to not have to concern about the existence of measures in most theoretical results (and it also provides a shorter notation).

There are two types of multiple orthogonal polynomials, type I and type II, with respect to systems of $r$ linear functionals, for some positive integer $r$.
Both types are labelled by multi-indices $\vec{n}=\left(n_1,\cdots,n_r\right)\in\N^r$ of norm $|\vec{n}|=n_1+\cdots+n_r$,
and both reduce to conventional orthogonal polynomials when $r=1$.

The type II multiple orthogonal polynomial with respect to a system of $r$ linear functionals $\left(v_1,\cdots,v_r\right)$ for the multi-index $\vec{n}=\left(n_1,\cdots,n_r\right)\in\N^r$ is a monic polynomial $P_{\vec{n}}(x)$ of degree $|\vec{n}|$ which satisfies the orthogonality conditions
\begin{align}
\label{type II orthogonality conditions}
\Functional{v_j}{x^kP_{\vec{n}}(x)}=0
\quad\text{if}\quad
0\leq k\leq n_j-1,
\quad\text{for each}\quad 
1\leq j\leq r.
\end{align}

The type I multiple orthogonal polynomials with respect to a system of $r$ linear functionals $\left(v_1,\cdots,v_r\right)$ for the multi-index $\vec{n}=\left(n_1,\cdots,n_r\right)\in\N^r$ are given by a vector of $r$ polynomials $\left(A_{\vec{n},1}(x),\cdots,A_{\vec{n},r}(x)\right)$ such that $\deg\left(A_{\vec{n},j}(x)\right)\leq n_j-1$ for all $1\leq j\leq r$ and the type I linear form $q_{\vec{n}}$, defined by
\begin{align}
\label{type I linear form definition}
q_{\vec{n}}=\sum_{j=1}^{r}A_{\vec{n},j}(x)\,v_j,
\end{align}
satisfies the orthogonality and normalisation conditions
\begin{align}
\label{type I orthogonality conditions}
\Functional{q_{\vec{n}}}{x^m}
=\sum_{j=1}^{r}\Functional{v_j}{x^m\,A_{\vec{n},j}(x)}
=\begin{cases}
0 & \text{ if } 0\leq m\leq|\vec{n}|-2, \\
1 & \text{ if } m=|\vec{n}|-1.
\end{cases}
\end{align}

In most relevant examples, the functionals $v_1,\cdots,v_r$ are induced by measures $\mu_1,\cdots,\mu_r$ which are all absolutely continuous with respect to a common measure $\mu$, 
i.e., there exist weight functions $w_1(x),\cdots,w_r(x)$ defined in the support of $\mu$ such that 
$\dis\Functional{v_j}{f}=\int f(x)w_j(x)\mathrm{d}\mu(x)$ for any $f\in\C[x]$ and any $1\leq j\leq r$. 
Then, the type I function $Q_{\vec{n}}(x)$ is defined by $\dis Q_{\vec{n}}(x)=\sum_{j=1}^{r}A_{\vec{n},j}(x)\,w_j(x)$ and 
the type I linear form $q_{\vec{n}}$ defined in \eqref{type I linear form definition} has the integral representation
$\dis\Functional{q_{\vec{n}}}{f}=\int f(x)Q_{\vec{n}}(x)\mathrm{d}\mu(x)$ for any $f\in\C[x]$.

The multiple orthogonality conditions in \eqref{type II orthogonality conditions} and \eqref{type I orthogonality conditions} give two systems of $|\vec{n}|$ equations for the $|\vec{n}|$ coefficients of the polynomials.
The matrices of these systems are the transpose of each other, so the existence and uniqueness of type I and type II multiple orthogonal polynomials are equivalent to each other.

We focus on the multi-indices on the step-line:
$\vec{n}=\left(n_1,\cdots,n_r\right)\in\N^r$ such that $n_1\geq n_2\geq\cdots\geq n_r\geq n_1-1$.
For $n\in\N$, there is a unique multi-index of norm $n$ on the step-line of $\N^r$.
If $n=rk+j$, with $k\in\N$ and $0\leq j\leq r-1$, that multi-index is
formed by $j$ entries equal to $k+1$ followed by $r-j$ entries equal to $k$.

The polynomial sequence $\seq[n\in\N]{P_n(x)}$ formed by the type II multiple orthogonal polynomials on the step-line of $\N^r$ is also said to be $r$-orthogonal.
A polynomial sequence $\seq[n\in\N]{P_n(x)}$ is $r$-orthogonal with respect to a system $\left(v_1,\cdots,v_r\right)$ if
\begin{align}
\label{type II orthogonality conditions on the step-line}
\Functional{v_j}{x^k\,P_n}
=\begin{cases}
N_n\neq 0 &\text{ if } n=rk+j-1 \\
  \hfil 0 &\text{ if } n\geq rk+j
\end{cases}
\quad\text{for each }j\in\{1,\cdots,r\}.
\end{align}

Similarly, we consider the sequences of type I multiple orthogonal polynomials $\big(A_{n,1}(x),\cdots,A_{n,r}(x)\big)_{n\in\N}$ and type I linear forms $\seq[n\in\N]{q_n}$ on the step-line.
However, for any system $\left(v_1,\cdots,v_r\right)$, the conditions on the degrees of the type I multiple orthogonal polynomials for the multi-index $\vec{0}=\left(0,\cdots,0\right)\in\N^r$ imply that
$A_{\vec{0},j}(x)=0$ for all $1\leq j\leq r$ and, consequently, $q_{\vec{0}}=0$.

Therefore, we use the notation $A_{n,j}(x)$ and $q_n$, with $n\in\N$ and $1\leq j\leq r$, for the type I multiple orthogonal polynomials and type I linear forms for the multi-index on the step-line with norm $n+1$, so that the sequences $\big(A_{n,1}(x),\cdots,A_{n,r}(x)\big)_{n\in\N}$ and $\seq[n\in\N]{q_n}$ do not start with trivial zero elements.
Using this notation:
\begin{itemize}
\item 
the type I orthogonality and normalization conditions in \eqref{type I orthogonality conditions} on the step-line become
\begin{align}
\label{type I orthogonality conditions on the step-line}
\Functional{q_n}{x^m}
=\sum_{j=1}^{r}\Functional{v_j}{x^m\,A_{n,j}(x)}
=\begin{cases}
0 & \text{ if } 0\leq m\leq n-1, \\
1 & \text{ if } m=n,
\end{cases}
\end{align}

\item 
the sequence of type I linear forms $\seq[k\in\N]{q_k}$ on the step-line is the dual sequence of the $r$-orthogonal polynomial sequence $\seq[n\in\N]{P_n(x)}$ with respect to the same system, i.e.,
\begin{align}
\label{biorthogonality conditions}
\Functional{q_k}{P_n(x)}
=\delta_{k,n}
=\begin{cases}
1 & \text{ if } n=k, \\
0 & \text{ if } n\neq k.
\end{cases}
\end{align}
This duality corresponds to the biorthogonality conditions in \cite[Thm.~23.1.6]{IsmailBook} reduced to the step-line.
\end{itemize}

A monic polynomial sequence $\seq[n\in\N]{P_n(x)}$ is $r$-orthogonal if and only if it satisfies a recurrence relation of order $r+1$ of the form
\begin{align}
\label{recurrence relation r-OP}
P_{n+1}(x)=x\,P_n(x)-\sum_{k=0}^{\min(r,n)}\gamma_{n-k}^{\,[k]}\,P_{n-k}(x),
\end{align}
with $\gamma_n^{\,[r]}\neq 0$ for all $n\in\N$ and the trivial initial condition $P_0(x)=1$.
The recurrence relation \eqref{recurrence relation r-OP} satisfied by a $r$-orthogonal polynomial sequence $\seq[n\in\N]{P_n(x)}$ can be expressed as a vector recurrence relation 
\begin{align}
x\,\left[P_n(x)\right]_{n\in\N}=\mathrm{H}\,\left[P_n(x)\right]_{n\in\N},
\end{align}
where $\left[P_n(x)\right]_{n\in\N}=\big[P_0(x),P_1(x),P_2(x),\cdots\big]^\top$ and the Hessenberg recurrence matrix is
\begin{align}
\label{Hessenberg matrix for r-OP}
\mathrm{H}
=\begin{bmatrix} 
\gamma^{[0]}_0 & 1 
\vspace*{0,1 cm}\\
\gamma^{[1]}_0 & \gamma^{[0]}_1 & 1 
\vspace*{0,1 cm}\\
\vdots & \ddots & \ddots & \ddots 
\vspace*{0,1 cm}\\ 
\gamma^{[r]}_0 & \cdots & \gamma^{[1]}_{r-1} & \gamma^{[0]}_r & 1 
\\
& \ddots & & \ddots & \ddots
\end{bmatrix}.
\end{align}

Furthermore, for any $n\geq 1$, $P_n(x)$ is the characteristic polynomial of the truncated $n$-by-$n$ Hessenberg matrix $\mathrm{H}_n$ formed by the first $n$ rows and columns of $\mathrm{H}$.
As a result, the study of the spectral properties of the matrices $\mathrm{H}_n$ gives information about the zeros of the $r$-orthogonal polynomials $P_n(x)$.

We are interested in bidiagonal matrix factorisations of the form $\mathrm{H}=\mathrm{L}_1\cdots\mathrm{L}_r\,\mathrm{U}$ as in \eqref{bidiagonal matrix factorisation L_1...L_r U introduction} for the Hessenberg matrix $\mathrm{H}$ in \eqref{Hessenberg matrix for r-OP}.
Observe that these bidiagonal matrix factorisations can be used to explicitly write the coefficients of the recurrence relation \eqref{recurrence relation r-OP} satisfied by $\seq[n\in\N]{P_n(x)}$ as sums of products of the nontrivial entries of the bidiagonal matrices $\mathrm{L}_1,\cdots,\mathrm{L}_r$, and $\mathrm{U}$ in \eqref{bidiagonal matrices U and L_k}:
\begin{align}
\label{recurrence coefficients as a combination of BCF coefficients}
\gamma_n^{\,[k]}=\sum_{r\geq\ell_0>\cdots>\ell_k\geq 0}\,\prod_{i=0}^{k}\alpha_{(r+1)(n+i)+\ell_i-r}
\quad\text{for any }n\in\N\text{ and }0\leq k\leq r,
\end{align}
with $\alpha_m=0$ for any $m<0$ (see \cite[Thm.~4.5]{MOPassociatedwithBCFforRatiosOfHypergeometricSeries}).

\subsection{Gauss-Borel factorisation of the moment matrix and multiple orthogonality}
\label{Gauss-Borel factorisation and multiple orthogonality}
\hspace*{\fill}

The moment matrix of a system of linear functionals $\left(v_1,\cdots,v_r\right)$ is the infinite matrix $\mathrm{M}=\left(m_{n,\ell}\right)_{n,\ell\in\N}$ with entries
$m_{n,rk+j}=\Functional{v_{j+1}}{x^{k+n}}$ for all $k,n\in\N$ and $0\leq j\leq r-1$.
More explicitly,
\begin{align}
\label{moment matrix explicit}
\mathrm{M}=
\begin{bmatrix}
 \Functional{v_1}{1} & \cdots & \Functional{v_r}{1} & 
 \Functional{v_1}{x} & \cdots & \Functional{v_r}{x} & 
\Functional{v_1}{x^2} & \cdots & \Functional{v_r}{x^2} & \cdots 
\vspace*{0,1 cm}\\ 
 \Functional{v_1}{x} & \cdots & \Functional{v_r}{x} & 
\Functional{v_1}{x^2} & \cdots & \Functional{v_r}{x^2} & 
\Functional{v_1}{x^3} & \cdots & \Functional{v_r}{x^3} & \cdots 
\vspace*{0,1 cm} \\
\Functional{v_1}{x^2} & \cdots & \Functional{v_r}{x^2} & 
\Functional{v_1}{x^3} & \cdots & \Functional{v_r}{x^3} & 
\Functional{v_1}{x^4} & \cdots & \Functional{v_r}{x^4} & \cdots 
\\ 
\vdots & & \vdots & \vdots & & \vdots & \vdots & & \vdots &
\end{bmatrix}.
\end{align}

The Gauss-Borel factorisation of the moment matrix $\mathrm{M}$ is the LU-decomposition $\mathrm{M}=\mathrm{C}\mathrm{D}$, where $\mathrm{C}$ is a unit-lower-triangular matrix and $\mathrm{D}$ is a non-singular upper-triangular matrix, which is unique if it exists. 
The Gauss-Borel factorisation exists if and only if $\Delta_n\neq 0$ for all $n\geq 1$, where $\Delta_n$ is the determinant of the $n$-by-$n$ leading principal submatrix of $\mathrm{M}$.
This condition is equivalent to the existence and uniqueness of the type I and type II multiple orthogonal polynomials on the step-line with respect to $\left(v_1,\cdots,v_r\right)$.

To find a more convenient expression for the moment matrix $\mathrm{M}$ in \eqref{moment matrix explicit}, we consider the infinite vectors
\begin{align}
\chi(x)= 
\begin{bmatrix}
1 \\
x \\
x^2 \\
\vdots
\end{bmatrix}
\quad\text{and}\quad
\chi_j(x)= 
\begin{bmatrix}
e_j \\
x \, e_j \\
x^2 \, e_j \\
\vdots
\end{bmatrix}
\quad\text{for }1\leq j\leq r,
\end{align}
where $\{e_j\}_{j=1}^r$ is the canonical basis in $\R^r$. 
Using these vectors, the moment matrix $\mathrm{M}$ can be given as
\begin{align}
\label{moment matrix vector rep.}
\mathrm{M}=\sum_{j=1}^{r}\Functional{v_j}{\chi(x)\chi_j^\top(x)},
\end{align}
where a functional acting on a matrix correspond to acting on each entry of that matrix.

To make this representation clearer, we present the illustrative example $r=2$. 
Then, 
\begin{align}
\chi_1^\top(x)=\left[1,0,x,0,x^2,0,\cdots\right]
\quad\text{and}\quad
\chi_2^\top(x)=\left[0,1,0,x,0,x^2,\cdots\right],
\end{align}
so
\begin{align}
\chi(x)\chi_1^\top(x)=
\begin{bmatrix}
 1 & 0 & x & 0 & \cdots \\
 x & 0 & x^2 & 0 & \cdots \\
x^2 & 0 & x^3 & 0 & \cdots \\
\vdots & \vdots & \vdots & \vdots
\end{bmatrix}
\quad\text{and}\quad
\chi(x)\chi_2^\top(x)=
\begin{bmatrix}
0 & 1 & 0 & x & \cdots \\
0 & x & 0 & x^2 & \cdots \\
0 & x^2 & 0 & x^3 & \cdots \\
\vdots & \vdots & \vdots & \vdots
\end{bmatrix}.
\end{align}
Therefore, $\Functional{v_1}{\chi(x)\chi_1^\top(x)}+\Functional{v_2}{\chi(x)\chi_2^\top(x)}$ is equal to the moment matrix $\mathrm{M}$ in \eqref{moment matrix explicit} with $r=2$.
The deduction of \eqref{moment matrix vector rep.} for generic $r$ is completely analogous to the case $r=2$.

The following result, which we do not prove here as it can be found, in more generality, in the context of mixed multiple orthogonal polynomials in \cite[Prop.~3]{MOPmixed&Gauss-Borel}, shows how the Gauss-Borel factorisation of the moment matrix can be used to obtain explicit expressions for type II multiple orthogonal polynomials and type I linear forms along the step-line.
\begin{proposition}\label{MOP representations from the Gauss-Borel factorisation prop.}
Suppose that the moment matrix $\mathrm{M}$ in \eqref{moment matrix explicit} admits a Gauss-Borel factorisation $\mathrm{M}=\mathrm{C}\mathrm{D}$, where $\mathrm{C}$ is a unit-lower-triangular matrix and $\mathrm{D}$ is a non-singular upper-triangular matrix.
Then, all the multiple orthogonal polynomials on the step-line with respect to $\left(v_1,\cdots,v_r\right)$ exist and are unique and the sequences of type II multiple orthogonal polynomials $\seq[n\in\N]{P_n(x)}$ and type I linear forms $\seq[n\in\N]{q_n}$ on the step-line are determined by
\begin{align}
\label{MOP representations from the Gauss-Borel factorisation}
\mathrm{P}(x)
=\big[P_0(x),P_1(x),P_2(x),\cdots\big]^\top
=\mathrm{C}^{-1}\,\chi(x)
\quad\text{and}\quad
q=\big[q_0,q_1,q_2,\cdots\big]
=\sum_{j=1}^{r}\chi_j^\top(x)\mathrm{D}^{-1}v_j.
\end{align}
\end{proposition}

The latter result means that the entries of the inverses of the triangular matrices appearing in the Gauss-Borel factorisation $\mathrm{M}=\mathrm{C}\mathrm{D}$ are the coefficients of the type II and type I multiple orthogonal polynomials on the step-line with respect to $\left(v_1,\cdots,v_r\right)$, 
which we denote by $\seq{P_n(x)}$ and $\seq{A_{n,1}(x),\cdots,A_{n,r}(x)}$, respectively.
To make this statement clearer, we explicitly write these polynomials as
\begin{align} 
P_n(x)=\sum_{k=0}^{n}p_n(k)\,x^k
\quad\text{and}\quad
A_{n,i}(x)=\sum_{\ell=0}^{n_i-1}a_{n,i}(\ell)\,x^{\ell},
\quad\text{for }n\in\N\text{ and }1\leq i\leq r,
\end{align}
where, for $n=rm+k$, with $m\in\N$ and $0\leq k\leq r-1$, we have
$n_i=m+1$ if $i\leq k+1$ and $n_i=m$ if $i\geq k+2$.
Then, $\mathrm{C}^{-1}$ and $\mathrm{D}^{-1}$ are triangular matrices such that:
\begin{itemize}
\item 
$\left(\mathrm{C}^{-1}\right)_{n,k}=p_n(k)$ when $k\leq n$,

\item 
$\left(\mathrm{D}^{-1}\right)_{m,n}=a_{n,i}(\ell)$
when $m=r\ell+i-1\leq n$, with $\ell\in\N$ and $1\leq i\leq r$.
\end{itemize}

We can get the recurrence relations verified by the type II polynomials and by the type I linear forms from the Gauss-Borel factorisation $\mathrm{M}=\mathrm{C}\mathrm{D}$ of the moment matrix \eqref{moment matrix explicit}. 
We omit the proof, as it follows directly from a more general result on mixed multiple orthogonal polynomials given in \cite[Prop.~13]{MOPmixed&Gauss-Borel}.
\begin{proposition}
\label{rec. rel. via Gauss-Borel}
Suppose that the moment matrix $\mathrm{M}$ in \eqref{moment matrix explicit} admits the Gauss-Borel factorisation $\mathrm{M}=\mathrm{C}\mathrm{D}$, where $\mathrm{C}$ is a unit-lower-triangular matrix and $\mathrm{D}$ is a non-singular upper-triangular matrix.
Then, the sequences of type II multiple orthogonal polynomials $\seq[n\in\N]{P_n(x)}$ and type I linear forms $\seq[n\in\N]{q_n}$ on the step-line with respect to $\left(v_1,\cdots,v_r\right)$ satisfy the recurrence relations
\begin{align}
\label{recurrence relations from the Gauss-Borel factorisation}
\mathrm{H}\,\mathrm{P}(x)=x\,\mathrm{P}(x)
\quad\text{and}\quad
q\,\mathrm{H}=x\,q,
\end{align}
where the vectors $\mathrm{P}(x)$ and $q$ are defined in \eqref{MOP representations from the Gauss-Borel factorisation} and $\mathrm{H}$ is the banded unit-lower-Hessenberg matrix with $r$ subdiagonals which satisfies
\begin{align}
\label{Hessenberg matrix 2 representations}
\mathrm{H}
=\mathrm{C}^{-1}\Lambda\mathrm{C}
=\mathrm{D}\left(\Lambda^\top\right)^r\mathrm{D}^{-1},
\end{align}
where $\Lambda$ is the infinite matrix with ones in the supradiagonal and zeros everywhere else, i.e.,
\begin{align}
\label{matrix Lambda}
\Lambda=
\begin{bmatrix}
0 & 1 & 0 & 0 & \cdots 
\\
0 & 0 & 1 & 0 & \ddots 
\\
0 & 0 & 0 & 1 & \ddots 
\\
\vdots & \ddots & \ddots & \ddots & \ddots 
\end{bmatrix}.
\end{align}
\end{proposition}

The factorisations \eqref{Hessenberg matrix 2 representations} of the Hessenberg matrix $\mathrm{H}$ will be used to construct its bidiagonal factorisation in the next section.


\subsection{Lattice paths, branched continued fractions, and production matrices}
\hspace*{\fill}

The branched continued fractions appearing here were introduced in \cite{AlanEtAl-LPandBCF1} as generating functions of $r$-Dyck paths and were further explored in \cite{AlanEtAl-LPandBCF2,AlanEtAl-LPandBCF3}. 
We follow their terminology and definitions.
An equivalent object was also introduced in \cite{AlbenqueBouttier-Constellations} and named as multicontinued fractions.

For a positive integer $r$, a {$r$-Dyck path} is a path in the lattice $\N^2$, starting and ending on the horizontal axis, using steps $(1,1)$, called {rises}, and $(1,-r)$, called {$r$-falls}.
More generally, a {partial $r$-Dyck path} in $\N^2$ is allowed to start and end anywhere in $\N^2$, using steps $(1,1)$ and $(1,-r)$.

We consider weighted (partial) $r$-Dyck paths, where each rise has weight $1$ and each $r$-fall to height $i$ has weight $\alpha_i$ and the weight of a path is the product of the weights of its steps. 
Moreover, the generating polynomial of a set of lattice paths is the sum of the weights of the paths in the set.
Observe that the length of a $r$-Dyck path is always a multiple of $r+1$. 
For an infinite sequence of weights $\seq[k\in\N]{\alpha_k}$, 
the {$r$-Stieltjes-Rogers polynomial} $\StieltjesRogersPoly[r]{n}{\boldsymbol{\alpha}}$, with $n\in\N$, 
is the generating polynomial of the $r$-Dyck paths from $(0,0)$ to $\big((r+1)n,0\big)$. 
They are an extension of the Stieltjes-Rogers polynomials introduced by Flajolet in \cite{FlajoletContinuedFractions}, which correspond to the case $r=1$.
%
See Fig.\ref{2-Dyck paths of length 6 fig.} for an illustrative example of how to compute $\StieltjesRogersPoly[2]{2}{\boldsymbol{\alpha}}$.
\begin{figure}[h]
\centering
\begin{tikzpicture}[scale=0.5]
\draw[->,color=black] (0,0) -- (0,5);
\draw[->,color=black] (0,0) -- (7,0);
\filldraw[black] (0,0) circle (4pt);
\draw[-,color=black] (0,0) -- (1,1);
\filldraw[black] (1,1) circle (4pt);
\draw[-,color=black] (1,1) -- (2,2);
\filldraw[black] (2,2) circle (4pt);
\draw[-,color=black] (2,2) -- (3,0);
\node[] at (3,1) {$\alpha_0$};
\filldraw[black] (3,0) circle (4pt);
\draw[-,color=black] (3,0) -- (4,1);
\filldraw[black] (4,1) circle (4pt);
\draw[-,color=black] (4,1) -- (5,2);
\filldraw[black] (5,2) circle (4pt);
\draw[-,color=black] (5,2) -- (6,0);
\node[] at (6,1) {$\alpha_0$};
\filldraw[black] (6,0) circle (4pt);
\end{tikzpicture}
\quad
\begin{tikzpicture}[scale=0.5]
\draw[->,color=black] (0,0) -- (0,5);
\draw[->,color=black] (0,0) -- (7,0);
\filldraw[black] (0,0) circle (4pt);
\draw[-,color=black] (0,0) -- (1,1);
\filldraw[black] (1,1) circle (4pt);
\draw[-,color=black] (1,1) -- (2,2);
\filldraw[black] (2,2) circle (4pt);
\draw[-,color=black] (2,2) -- (3,3);
\filldraw[black] (3,3) circle (4pt);
\draw[-,color=black] (3,3) -- (4,1);
\node[] at (4,2) {$\alpha_1$};
\filldraw[black] (4,1) circle (4pt);
\draw[-,color=black] (4,1) -- (5,2);
\filldraw[black] (5,2) circle (4pt);
\draw[-,color=black] (5,2) -- (6,0);
\node[] at (6,1) {$\alpha_0$};
\filldraw[black] (6,0) circle (4pt);
\end{tikzpicture}
\quad
\begin{tikzpicture}[scale=0.5]
\draw[->,color=black] (0,0) -- (0,5);
\draw[->,color=black] (0,0) -- (7,0);
\filldraw[black] (0,0) circle (4pt);
\draw[-,color=black] (0,0) -- (1,1);
\filldraw[black] (1,1) circle (4pt);
\draw[-,color=black] (1,1) -- (2,2);
\filldraw[black] (2,2) circle (4pt);
\draw[-,color=black] (2,2) -- (3,3);
\filldraw[black] (3,3) circle (4pt);
\draw[-,color=black] (3,3) -- (4,4);
\filldraw[black] (4,4) circle (4pt);
\draw[-,color=black] (4,4) -- (5,2);
\node[] at (5,3) {$\alpha_2$};
\filldraw[black] (5,2) circle (4pt);
\draw[-,color=black] (5,2) -- (6,0);
\node[] at (6,1) {$\alpha_0$};
\filldraw[black] (6,0) circle (4pt);
\end{tikzpicture}
\caption{All $2$-Dyck paths from $(0,0)$ to $(6,0)$
$\implies\StieltjesRogersPoly[2]{2}{\boldsymbol{\alpha}}
=\alpha_0^2+\alpha_0\alpha_1+\alpha_0\alpha_2
$.}
\label{2-Dyck paths of length 6 fig.}
\end{figure}
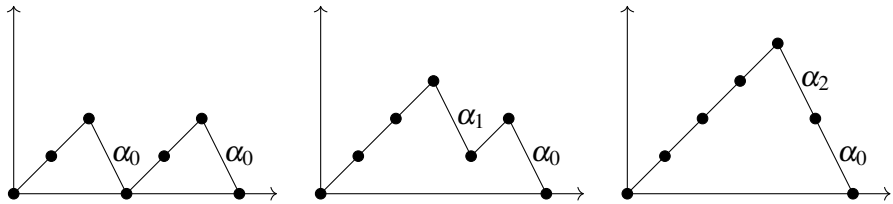

A partial $r$-Dyck path starting at $(0,0)$ only passes by points $(x,y)\in\N^2$ such that $y=x-(r+1)m$, where $m\in\N$ is the number of $r$-falls of the path and $x$ is the total number of steps.
We are interested in the following generalisations of the $r$-Stieltjes-Rogers polynomials
(introduced in \cite{AlanEtAl-LPandBCF1} and \cite{AlanEtAl-LPandBCF3}, respectively):
\begin{itemize}
\item 
the {modified $r$-Stieltjes-Rogers polynomials of type $j$}, 
denoted by $\modifiedStieltjesRogersPoly[r]{n}{j}{\boldsymbol{\alpha}}$ with $n,j\in\N$, 
are the generating polynomials of the partial $r$-Dyck paths from $(0,0)$ to $\big((r+1)n+j,j\big)$,
\item 
the {generalised $r$-Stieltjes-Rogers polynomials of type $j$}, 
denoted by $\generalisedStieltjesRogersPolyTypeJ[r]{n}{k}{j}{\boldsymbol{\alpha}}$ with $n,k,j\in\N$, 
are the generating polynomials of the partial $r$-Dyck paths from $(0,0)$ to $\big((r+1)n+j,(r+1)k+j\big)$.
\end{itemize}

Let $f_0(t)$ be the generating function for $r$-Dyck paths, considered as a formal power series in $t$:
\begin{align}
f_0(t)=\sum_{n=0}^{\infty}\StieltjesRogersPoly{n}{\boldsymbol{\alpha}}t^n.
\end{align}
More generally, for any $k\in\N$, let $f_k(t)$ be the generating function for partial $r$-Dyck paths which start and end at height $k$ without ever going below it 
(i.e., $f_k(t)$ is $f_0(t)$ with each $\alpha_i$ replaced by $\alpha_{i+k}$).
The sequence $\seq[k\in\N]{f_k(t)}$ satisfies the functional equation (see \cite[Eqs.~2.26-2.27]{AlanEtAl-LPandBCF1})
\begin{align}
\label{functional equation for the generating function of r-Dyck paths}
f_k(t)=1+\alpha_k\,t\prod_{j=0}^{r}f_{k+j}(t)
\Leftrightarrow
f_k(t)=\frac{1}{\dis 1-\alpha_k\,t\prod_{j=1}^{r}f_{k+j}(t)}
\quad\text{for all }k\in\N.
\end{align}
Successively iterating the former we find that (see \cite[Eqs.~2.28]{AlanEtAl-LPandBCF1})
\begin{align}
\label{branched cont. frac. of Stieltjes type}
f_k(t)
=\cfrac{1}{\dis 1-\alpha_k\,t\prod_{i_1=1}^{r}\cfrac{1}{\dis 1-\alpha_{k+i_1}\,t\prod_{i_2=1}^{r}\cfrac{1}{\dis 1-\alpha_{k+i_1+i_2}\,t\prod_{i_3=1}^{r}\cfrac{1}{\dis 1-\cdots}}}}
\quad\text{for all }k\in\N.
\end{align}

The right-hand side of \eqref{branched cont. frac. of Stieltjes type} is known as a {(Stieltjes-type) $r$-branched continued fraction}, because it reduces to a classical Stieltjes continued fraction (also known as an S-fraction) when $r=1$. 

As explained in \cite{AlanSokal.MOP-ProdMat-BCF}, to connect lattice paths and branched continued fractions with multiple orthogonal polynomials, it is fundamental to use the production matrices originally introduced in \cite{ProductionMatrices2005}.

Let $\mathrm{H}=\left(h_{i,j}\right)_{i,j\in\N}$ be an infinite matrix with entries in a commutative ring $R$.
If all the powers of $\mathrm{H}$ are well-defined (which is the case when $\mathrm{H}$ is an Hessenberg matrix), we can define an infinite matrix $\mathrm{S}=\left(s_{n,k}\right)_{n,k\in\N}$ by $s_{n,k}=\left(\mathrm{H}^n\right)_{0,k}$ for all $n,k\in\N$.
We call $\mathrm{H}$ the {production matrix} and $\mathrm{S}$ the {output matrix}.



\section{Bidiagonal factorisations linked to the Gauss-Borel factorisation}
\label{Factorisations using MOP coefficients}
\hspace*{\fill}

For a system of linear functionals $\left(v_1,\cdots,v_r\right)$, we consider its Christoffel transformations
\begin{align}
\label{systems V^[j]}
V^{[j]}
=\left(v_{j+1},\cdots,v_r,x\,v_1,\cdots,x\,v_j\right),\;
0\leq j\leq r.
\end{align}
In particular, $V^{[0]}=\left(v_1,\cdots,v_r\right)$ and $V^{[r]}=\left(x\,v_1,\cdots,x\,v_r\right)$.

In this section, we show that the existence and uniqueness of the multiple orthogonal polynomials on the step-line with respect to all systems $V^{[j]}$ in \eqref{systems V^[j]} is a sufficient condition to obtain bidiagonal factorisations of the Hessenberg recurrence matrices of these systems, with all nontrivial entries of the bidiagonal matrices different from zero.
In particular, we obtain a factorisation $\mathrm{H}=\mathrm{L}_1\cdots\mathrm{L}_r\,\mathrm{U}$ of the Hessenberg recurrence matrix of the system $\left(v_1,\cdots,v_r\right)$.

We start by using the Gauss-Borel factorisation of the moment matrices of the systems $V^{[j]}$ to express the nontrivial entries of the bidiagonal matrices appearing in our factorisations.
Then, we rewrite the formulas for the nontrivial entries of the bidiagonal matrices, firstly using coefficients of the type I and type II multiple orthogonal polynomials on the step-line with respect to the systems $V^{[j]}$ instead of their Gauss-Borel factorisation, and later using the determinants of the finite leading principal submatrices of the moment matrices of the systems $V^{[j]}$.

The following theorem is the main theoretical result of this section, where we show that the existence and uniqueness of the multiple orthogonal polynomials on the step-line with respect to all systems $V^{[j]}$ in \eqref{systems V^[j]} is a sufficient condition to obtain bidiagonal factorisations of the Hessenberg recurrence matrices of these systems, using the entries of either one of the triangular matrices appearing in the Gauss-Borel factorisation of the moment matrices of the systems $V^{[j]}$.
\begin{theorem}
\label{bidiagonal fatorisations using the Gauss-Borel fatorisation thm.}
For a system of linear functionals $\left(v_1,\cdots,v_r\right)$ such that the multiple orthogonal polynomials on the step-line with respect to all systems $V^{[j]}$, $0\leq j\leq r$, in \eqref{systems V^[j]} exist and are unique, 
let $\dis{\mathrm{M}^{[j]}=\mathrm{C}^{[j]}\mathrm{D}^{[j]}}$ be the Gauss-Borel factorisation of the moment matrix $\mathrm{M}^{[j]}$ of the system $V^{[j]}$ for each $0\leq j\leq r$,
where ${\mathrm{C}^{[j]}=\left(c^{[j]}_{i,k}\right)_{i,k\in\N}}$ and ${\mathrm{D}^{[j]}=\left(d^{[j]}_{i,k}\right)_{i,k\in\N}}$
are, respectively, a unit-lower-triangular matrix and a non-singular upper-triangular matrix.
Then, for any $0\leq j\leq r$, the Hessenberg recurrence matrix $\mathrm{H}^{[j]}$ of the system $V^{[j]}$ admits the bidiagonal factorisation
\begin{align}
\label{eq:bidiagonal_j}
\mathrm{H}^{[j]}=\mathrm{L}_{j+1}\cdots\mathrm{L}_r\mathrm{U}\mathrm{L}_{1}\cdots\mathrm{L}_j,
\end{align}
where $\mathrm{L}_1,\cdots,\mathrm{L}_r$, and $\mathrm{U}$ are bidiagonal matrices of the form in \eqref{bidiagonal matrices U and L_k}, with nontrivial entries $\seq[m\in\N]{\alpha_m}$ all different from zero such that, for all $n\in\N$,
\begin{align}
\label{alpha (r+1)n}
\alpha_{(r+1)n} 
=\frac{d^{[r]}_{n,n}}{d^{[0]}_{n,n}}=c^{[0]}_{n+1,n}-c^{[r]}_{n,n-1},
\quad\text{with }c^{[r]}_{0,-1}=0,
\end{align}
and
\begin{align}
\label{alpha (r+1)n+i, i from 1 to r}
\alpha_{(r+1)n+i} 
=\frac{d^{[i-1]}_{n+1,n+1}}{d^{[i]}_{n,n}}
=c^{[i]}_{n+1,n}-c^{[i-1]}_{n+1,n} 
\quad\text{for }1\leq i\leq r.
\end{align}
In particular, the Hessenberg matrix $\mathrm{H}=\mathrm{H}^{[0]}$ admits the bidiagonal factorisation \eqref{bidiagonal matrix factorisation L_1...L_r U introduction}.
\end{theorem}

To prove this theorem, we need the following auxiliary result.
\begin{lemma}\label{Christoffel transformations prop.}
Let $\mathrm{M}$ be the moment matrix in \eqref{moment matrix explicit} and $\Lambda$ be the matrix given in \eqref{matrix Lambda}. 
Then:
\begin{enumerate}
\item
$\mathrm{M}\left(\Lambda^\top\right)^j $ is the moment matrix of the system $V^{[j]}$
for all $1\leq j\leq r$,

\item
$\Lambda\mathrm{M}$ is the moment matrix of the system $V^{[r]}$.
\end{enumerate}
\end{lemma}

\begin{proof}[Proof of Lemma \ref{Christoffel transformations prop.}]
Recalling the formula \eqref{moment matrix vector rep.} for the moment matrix $\mathrm{M}$ of $\left(v_1,\cdots,v_r\right)$, we have
\begin{align}
\mathrm{M}\Lambda^\top
=\sum_{i=1}^{r} \Functional{v_i}{\chi(x)\chi_i^\top(x)}\Lambda^\top
=\sum_{i=1}^{r}\Functional{v_i}{\chi(x)\big(\Lambda\chi_i(x)\big)^\top}.
\end{align}
It is clear from the definition \eqref{matrix Lambda} of $\Lambda$ that 
$\Lambda\chi_1(x)=x\,\chi_{r}(x)$ and
$\Lambda\chi_j(x)=\chi_{j-1}(x)$ for $2\leq j\leq r$.
Hence,
\begin{align}
\mathrm{M}\Lambda^\top
=\sum_{i=2}^{r}\Functional{v_i}{\chi(x)\chi_{i-1}^\top(x)}
+\Functional{x\,v_1}{\chi(x)\chi_r^\top(x)}
\end{align}
Therefore, $\mathrm{M}\Lambda^\top$ is the moment matrix of the system $V^{[1]}=\left(v_2,\cdots,v_r,x\,v_1\right)$. 
Iterating this procedure, we find that $\mathrm{M}\left(\Lambda^\top\right)^j$ is the moment matrix of $V^{[j]}=\left(v_{j+1},\cdots,v_r,x\,v_1,\cdots,x\,v_j\right)$ for all $1\leq j\leq r$, proving $(i)$.
In particular, 
the moment matrix of $V^{[r]}=\left(x\,v_1,\cdots,x\,v_r\right)$ is
\begin{equation}
\mathrm{M}\left(\Lambda^\top\right)^r
=\sum_{i=1}^{r}\Functional{v_i}{\chi(x)\,\chi_i^\top(x)}\big(\Lambda^r\big)^\top
=\sum_{i=1}^{r}\Functional{v_i}{\chi(x)\Big(\Lambda^r\chi_i(x)\Big)^\top}.
\end{equation}
We can observe from \eqref{matrix Lambda} that $\Lambda^r\chi_i(x)=x\,\chi_i(x)$ and $\Lambda\chi(x)=x\,\chi(x)$.
As a result,
\begin{align}
\mathrm{M}\left(\Lambda^\top\right)^r
=\sum_{i=1}^{r}\Functional{v_i}{x\,\chi(x)\chi_i^\top(x)}
=\sum_{i=1}^{r}\Functional{v_i}{\Lambda\,\chi(x)\chi_i^\top(x)}
=\Lambda\mathrm{M}.
\end{align}
Therefore, we conclude that the moment matrix of $\left(x\,v_1,\cdots,x\,v_r\right)$ is $\Lambda\mathrm{M}$, proving $(ii)$.
\end{proof}

\begin{proof}[Proof of Theorem \ref{bidiagonal fatorisations using the Gauss-Borel fatorisation thm.}]
Using Lemma \ref{Christoffel transformations prop.}, we know that the moment matrices of the systems $V^{[i]}$ and $V^{[i-1]}$ are related by 
$\mathrm{M}^{[i]}=\mathrm{M}^{[i-1]}\Lambda^\top$ for any $1\leq i\leq r$. 
Applying the Gauss-Borel factorisation to both $\mathrm{M}^{[i]}$ and $\mathrm{M}^{[i-1]}$, we obtain
$\mathrm{C}^{[i]}\mathrm{D}^{[i]}=\mathrm{C}^{[i-1]}\mathrm{D}^{[i-1]}\Lambda^\top$.
Therefore, we can define 
\begin{align}
\label{definition of L_i}
\mathrm{L}_i
=\left(\mathrm{C}^{[i-1]}\right)^{-1}\mathrm{C}^{[i]}
=\mathrm{D}^{[i-1]}\Lambda^\top\left(\mathrm{D}^{[i]}\right)^{-1} 
\quad\text{for all }1\leq i\leq r.
\end{align}

Because $\mathrm{C}^{[i-1]}$ and $\mathrm{C}^{[i]}$ are unit-lower-triangular matrices, 
$\mathrm{L}_i$ is also unit-lower-triangular.
Moreover, because $\mathrm{D}^{[i-1]}$ and $\mathrm{D}^{[i]}$ are non-singular upper-triangular matrices and all nonzero entries of $\Lambda^\top$ are in its first subdiagonal,
$\mathrm{L}_i$ is upper-Hessenberg, i.e., its first subdiagonal entries are nontrivial and all the entries below it are equal to $0$.
Therefore, $\mathrm{L}_i$ is a lower-bidiagonal matrix of the form in \eqref{bidiagonal matrices U and L_k} and we denote its subdiagonal entries by $\alpha_{(r+1)n+i}$, with $n\in\N$, as in \eqref{bidiagonal matrices U and L_k}.

Recalling Proposition \ref{rec. rel. via Gauss-Borel}, 
we know that 
$\mathrm{H}^{[j]}=\mathrm{D}^{[j]}\left(\Lambda^\top\right)^r\left(\mathrm{D}^{[j]}\right)^{-1}$ 
for all $0\leq j\leq r$.
Taking into account \eqref{definition of L_i}, we have 
$\mathrm{D}^{[j]}\,\Lambda^\top=\mathrm{L}_{j+1}\,\mathrm{D}^{[j+1]}$ for any $0\leq j\leq r-1$.
Successively iterating this equation, we find that
$\mathrm{D}^{[j]}\left(\Lambda^\top\right)^{r-j}=\mathrm{L}_{j+1}\cdots\mathrm{L}_r\mathrm{D}^{[r]}$ for all $0\leq j\leq r$.
As a result, we derive that
\begin{align}
\label{first decomposition H^[j]}
\mathrm{H}^{[j]}
=\mathrm{L}_{j+1}\cdots\mathrm{L}_r\,\mathrm{D}^{[r]}
\left(\Lambda^\top\right)^j\left(\mathrm{D}^{[j]}\right)^{-1}
\quad\text{for all }0\leq j\leq r.
\end{align}
In particular, taking $j=0$, we have 
\begin{align}
\label{first decomposition H^[0]}
\mathrm{H}^{[0]}=\mathrm{L}_1\mathrm{L}_2\cdots\mathrm{L}_r\mathrm{D}^{[r]}\left(\mathrm{D}^{[0]}\right)^{-1}.
\end{align}

Now, we define $\mathrm{U}=\mathrm{D}^{[r]}\left(\mathrm{D}^{[0]}\right)^{-1}$ and we check that $\mathrm{U}$ is an upper-bidiagonal matrix of the form in \eqref{bidiagonal matrices U and L_k}.
%
By definition, $\mathrm{U}$ is a non-singular upper-triangular matrix.
Moreover, using Lemma \ref{Christoffel transformations prop.}, 
$\mathrm{M}^{[r]}=\Lambda\mathrm{M}^{[0]}$.
Thus, applying the Gauss-Borel factorisation,
$\mathrm{C}^{[r]}\mathrm{D}^{[r]}=\Lambda\mathrm{C}^{[0]}\mathrm{D}^{[0]}$.
Equivalently,
$\mathrm{D}^{[r]}=\left(\mathrm{C}^{[r]}\right)\Lambda\mathrm{C}^{[0]}\mathrm{D}^{[0]}$.
So, we derive the alternative expression
$\mathrm{U}=\left(\mathrm{C}^{[r]}\right)^{-1}\Lambda\mathrm{C}^{[0]}$.
Because $\mathrm{C}^{[0]}$ and $\mathrm{C}^{[r]}$ are unit-lower-triangular matrices, 
$\mathrm{U}$ is a lower-Hessenberg matrix whose supradiagonal entries are all equal to $1$. 

Hence, we have shown that $\mathrm{U}$ is an upper-bidiagonal matrix of the form in \eqref{bidiagonal matrices U and L_k} and we denote its diagonal entries by $\alpha_{(r+1)n}$, with $n\in\N$, as in \eqref{bidiagonal matrices U and L_k}.
Therefore, $\mathrm{H}=\mathrm{H}^{[0]}$ admits the bidiagonal factorisation $\mathrm{H}=\mathrm{L}_1\mathrm{L}_2\cdots\mathrm{L}_r\mathrm{U}$, which corresponds to \eqref{bidiagonal matrix factorisation L_1...L_r U introduction} and to the case $j=0$ of \eqref{eq:bidiagonal_j}.

Furthermore, recalling that $\mathrm{L}_i=\mathrm{D}^{[i-1]}\Lambda^\top\left(\mathrm{D}^{[i]}\right)^{-1}$, we have $\Lambda^\top\left(\mathrm{D}^{[i]}\right)^{-1}=\left(\mathrm{D}^{[i-1]}\right)^{-1}\mathrm{L}_i$ for all $1\leq i\leq r$.
Applying iteratively this identity to \eqref{first decomposition H^[j]}, we deduce that
\begin{align}
\mathrm{H}^{[j]}
=\mathrm{L}_{j+1}\cdots\mathrm{L}_r\,\mathrm{D}^{[r]}
\left(\Lambda^\top\right)^{j-1}\left(\mathrm{D}^{[j-1]}\right)^{-1}\mathrm{L}_j
=\cdots
=\mathrm{L}_{j+1}\cdots\mathrm{L}_r\,\mathrm{D}^{[r]}
\left(\mathrm{D}^{[0]}\right)^{-1}\mathrm{L}_1\cdots\mathrm{L}_j.
\end{align}
Taking into account that $\mathrm{U}=\mathrm{D}^{[r]}\left(\mathrm{D}^{[0]}\right)^{-1}$, we conclude that \eqref{eq:bidiagonal_j} holds for all $0\leq j\leq r$.

Now, we prove that the nontrivial entries of the bidiagonal matrices $\mathrm{L}_1,\cdots,\mathrm{L}_r$ and $\mathrm{U}$ satisfy \eqref{alpha (r+1)n}-\eqref{alpha (r+1)n+i, i from 1 to r}.
Because $\mathrm{D}^{[r]}$ and $\mathrm{D}^{[0]}$ are non-singular upper-triangular matrices, the diagonal entries of $\mathrm{U}=\mathrm{D}^{[r]}\left(\mathrm{D}^{[0]}\right)^{-1}$ can be expressed as
\begin{align}
\label{alpha (r+1)n, formula as a ratio}
\alpha_{(r+1)n}
=\left(\mathrm{U}\right)_{n,n}
=\left(\mathrm{D}^{[r]}\right)_{n,n}\left(\left(\mathrm{D}^{[0]}\right)^{-1}\right)_{n,n}
=\dfrac{d^{[r]}_{n,n}}{d^{[0]}_{n,n}}.
\end{align}

We can also find an alternative formula for $\alpha_{(r+1)n}$ using the expression
$\mathrm{U}=\left(\mathrm{C}^{[r]}\right)^{-1}\Lambda\,\mathrm{C}^{[0]}$.
Taking into account that $\mathrm{C}^{[r]}$ and $\mathrm{C}^{[0]}$ are unit-lower-triangular matrices and recalling the definition \eqref{matrix Lambda} of $\Lambda$, we compute the diagonal entries of this product and deduce that
$\alpha_0=\left(\mathrm{U}\right)_{0,0}=\left(\mathrm{C}^{[0]}\right)_{1,0}=a_{1,0}^{[0]}$ and
\begin{align}
\label{alpha (r+1)n, formula as a difference}
\alpha_{(r+1)n}
=\left(\mathrm{U}\right)_{n,n}
=\left(\left(\mathrm{C}^{[r]}\right)^{-1}\right)_{n,n-1}+\left(\mathrm{C}^{[0]}\right)_{n+1,n}
=c^{[0]}_{n+1,n}-c^{[r]}_{n,n-1}
\quad\text{for all }n\geq 1.
\end{align}
Therefore, \eqref{alpha (r+1)n} holds.

Analogously, we compute the subdiagonal entries of $\mathrm{L}_i$ using the expression $\mathrm{L}_i=\left(\mathrm{C}^{[i-1]}\right)^{-1}\mathrm{C}^{[i]}$ to obtain
\begin{align}
\label{alpha (r+1)n+i, i from 1 to r, formula as a difference}
\alpha_{(r+1)n+i}
=\left(\mathrm{L}_i\right)_{n+1,n}
=\left(\mathrm{C}^{[i]}\right)_{n+1,n}+\left(\left(\mathrm{C}^{[i-1]}\right)^{-1}\right)_{n+1,n}
=c^{[i]}_{n+1,n}-c^{[i-1]}_{n+1,n},
\end{align}
or using the expression $\mathrm{L}_i=\mathrm{D}^{[i-1]}\Lambda^\top\left(\mathrm{D}^{[i]}\right)^{-1}$ to obtain
\begin{align}
\label{alpha (r+1)n+i, i from 1 to r, formula as a ratio}
\alpha_{(r+1)n+i}
=\left(\mathrm{L}_i\right)_{n+1,n}
=\left(\mathrm{D}^{[i-1]}\right)_{n+1,n+1}\left(\left(\mathrm{D}^{[i-1]}\right)^{-1}\right)_{n,n}
=\frac{d^{[i-1]}_{n+1,n+1}}{d^{[i]}_{n,n}}.
\end{align}
As a result, \eqref{alpha (r+1)n+i, i from 1 to r} also holds.
Finally, observe that \eqref{alpha (r+1)n, formula as a ratio} and \eqref{alpha (r+1)n+i, i from 1 to r, formula as a ratio} imply that $\alpha_m\neq 0$ for all $m\in\N$, because each $\alpha_m$ is the ratio of two diagonal entries of two non-singular upper-triangular matrices.
\end{proof}

As explained after Proposition \ref{MOP representations from the Gauss-Borel factorisation prop.}, the entries of the inverses of the triangular matrices appearing in the Gauss-Borel factorisations $\dis{\mathrm{M}^{[j]}=\mathrm{C}^{[j]}\mathrm{D}^{[j]}}$ of the moment matrices of the systems $V^{[j]}$, $0\leq j\leq r$, are the coefficients of the type II and type I multiple orthogonal polynomials on the step-line with respect to the corresponding system.
As a result, we can rewrite \eqref{alpha (r+1)n} and \eqref{alpha (r+1)n+i, i from 1 to r}, using the coefficients of those multiple orthogonal polynomials.
For that purpose, we explicitly write the type II and type I multiple orthogonal polynomials on the step-line with respect to $V^{[j]}$, $0\leq j\leq r$, as 
\begin{align} 
P^{[j]}_n(x)=\sum_{k=0}^{n}p^{[j]}_n(k)\,x^k
\quad\text{and}\quad
A^{[j]}_{n,i}(x)=\sum_{\ell=0}^{n_i-1}a^{[j]}_{n,i}(\ell)\,x^{\ell},
\quad\text{for }n\in\N\text{ and }1\leq i\leq r,
\end{align}
where, for $n=rm+k$, with $m\in\N$ and $0\leq k\leq r-1$, we have
$n_i=m+1$ if $i\leq k+1$ and $n_i=m$ if $i\geq k+2$.
%

Then:
\begin{itemize}
\item 
the diagonal entries of $\left(\mathrm{D}^{[j]}\right)^{-1}$ are the leading coefficients of the type I multiple orthogonal polynomials $A^{[j]}_{rm+k,k+1}(x)$,
so $d^{[j]}_{rm+k,rm+k}=\left(a^{[j]}_{rm+k,k+1}(m)\right)^{-1}$ for all $m\in\N$, $0\leq k\leq r-1$, and $0\leq j\leq r$,

\item 
the subdiagonal entries of $\left(\mathrm{C}^{[j]}\right)^{-1}$ are the coefficients $p^{[j]}_{n+1}(n)$ of type II multiple orthogonal polynomials $P^{[j]}_{n+1}(x)$, 
so $c^{[j]}_{n+1,n}=-p^{[j]}_{n+1}(n)$ for all $n\in\N$ and $0\leq j\leq r$. 
\end{itemize}

Therefore, the result below follows directly from Theorem \ref{bidiagonal fatorisations using the Gauss-Borel fatorisation thm.}.
\begin{corollary}\label{formulas for the alphas using the MOP coeff}
Let $\left(v_1,\cdots,v_r\right)$ be a system of linear functionals such that the multiple orthogonal polynomials on the step-line with respect to all systems $V^{[j]}$, $0\leq j\leq r$, in \eqref{systems V^[j]} exist and are unique.
Then, for any $0\leq j\leq r$, the Hessenberg recurrence matrix $\mathrm{H}^{[j]}$ of the system $V^{[j]}$ admits the bidiagonal factorisation \eqref{eq:bidiagonal_j},
with nontrivial entries $\seq[\ell\in\N]{\alpha_{\ell}}$ which we can express using coefficients of the type I and type II multiple orthogonal polynomials as follows:
\begin{enumerate}
\item 
for all $m\in\N$, 
\begin{align}
\label{bidiagonal matrix entries using type I MOP 1}
\alpha_{(r+1)(rm+k)}  
=\frac{a^{[0]}_{rm+k,k+1}(m)}{a^{[r]}_{rm+k,k+1}(m)}
\quad\text{for any }0\leq k\leq r-1,
\end{align}
and, for $1\leq i\leq r$,
\begin{align}
\label{bidiagonal matrix entries using type I MOP 2}
\alpha_{(r+1)(rm+k)+i} 
=\frac{a^{[i]}_{rm+k,k+1}(m)}{a^{[i-1]}_{rm+k+1,k+2}(m)}
\text{ if }k\leq r-2
\quad\text{and}\quad 
\alpha_{(r+1)(rm+r-1)+i} 
=\frac{a^{[i]}_{rm+r-1,r}(m)}{a^{[i-1]}_{r(m+1),1}(m+1)};
\end{align}

\item 
for all $n\in\N$,
\begin{align}
\label{eq:alpha_tipoII}
\alpha_{(r+1)n}=p^{[r]}_{n}(n-1)-p^{[0]}_{n+1}(n)
\quad\text{and}\quad 
\alpha_{(r+1)n+i}=p^{[i-1]}_{n+1}(n)-p^{[i]}_{n+1}(n)
\quad\text{for } 
1\leq i\leq r.
\end{align}
\end{enumerate}
\end{corollary}

For $n\geq 1$ and $0\leq j\leq r$, we denote by $\Delta_n^{[j]}$ the determinant of the $n$-by-$n$ leading principal submatrix of the moment matrix of the system $\mathrm{V}^{[j]}$.
The coefficients of the multiple orthogonal polynomials on the step-line with respect to the systems $V^{[j]}$ can be written as ratios of the determinants $\Delta_n^{[j]}$ and determinants obtained from them by removing a specific row and column;
consequently, the same applies to the entries $d^{[j]}_{n,n}$ and $c^{[j]}_{n+1,n}$ in \eqref{alpha (r+1)n} and \eqref{alpha (r+1)n+i, i from 1 to r}.
Therefore, we find expressions for the nontrivial entries of the bidiagonal matrices in the factorisations of the Hessenberg recurrence matrices using those determinants, as follows.
\begin{corollary}\label{formulas for the alphas using determinants}
Let $\left(v_1,\cdots,v_r\right)$ be a system of linear functionals such that the multiple orthogonal polynomials on the step-line with respect to all systems $V^{[j]}$, $0\leq j\leq r$, in \eqref{systems V^[j]} exist and are unique.
Then, for any $0\leq j\leq r$, the Hessenberg recurrence matrix $\mathrm{H}^{[j]}$ of the system $V^{[j]}$ admits the bidiagonal factorisation \eqref{eq:bidiagonal_j},
with nontrivial entries $\seq[m\in\N]{\alpha_m}$, which, for $n\in\N$, can be expressed as
\begin{align}
\label{formulas for the alphas as ratios of 2-by-2 determinants}
\alpha_{(r+1)n}=\frac{\Delta_{n+1}^{[r]}\,\Delta_n^{[0]}}{\Delta_n^{[r]}\,\Delta_{n+1}^{[0]}}
\quad\text{and}\quad
\alpha_{(r+1)n+i}=\frac{\Delta_{n+2}^{[i-1]}\,\Delta_n^{[i]}}{\Delta_{n+1}^{[i-1]}\,\Delta_{n+1}^{[i]}}
\quad\text{for }1\leq i\leq r,
\end{align}
and as
\begin{align}
\label{formulas for the alphas as diferences of 2-by-2 determinants}
\alpha_{(r+1)n}=\frac{\widehat{\Delta}^{[0]}_{n+2}}{\Delta^{[0]}_{n+1}}-\frac{\widehat{\Delta}^{[r]}_{n+1}}{\Delta^{[r]}_n}
\quad\text{and}\quad
\alpha_{(r+1)n+i}=\frac{\widehat{\Delta}^{[i]}_{n+2}}{\Delta^{[i]}_{n+1}} -\frac{\widehat{\Delta}^{[i-1]}_{n+2}}{\Delta^{[i-1]}_{n+1}}
\quad\text{for }1\leq i\leq r,
\end{align}
where, for any $0\leq j\leq r$, 
$\dis\Delta_0^{[j]}=1$, $\widehat{\Delta}^{[j]}_1=0$, 
and for $n\geq 1$:
\begin{itemize}
\item 
$\Delta^{[j]}_n$ is the determinant of the $n$-by-$n$ leading principal submatrix $\mathrm{M}^{[j]}_n$ of the moment matrix of $\mathrm{V}^{[j]}$,
\item 
$\widehat{\Delta}^{[j]}_{n+1}$ is the determinant of the matrix obtained from $\mathrm{M}^{[j]}_{n+1}$ by deleting its $n$-th row and $(n+1)$-th column.
\end{itemize}
Moreover, if $\Delta^{[j-1]}_1=\Functional{v_j}{1}>0$ for all $1\leq j\leq r$, 
then $\alpha_m>0$ for all $m\in\N$ if and only if $\Delta^{[j]}_n>0$ for all $n\in\N$ and $0\leq j\leq r$.
\end{corollary}

Recall that the existence and uniqueness of the multiple orthogonal polynomials on the step-line with respect to all systems $V^{[j]}$, $0\leq j\leq r$, in \eqref{systems V^[j]} is equivalent to having $\Delta^{[j]}_n\neq 0$ for all $n\in\N$ and $0\leq j\leq r$.
As a result, the formulas \eqref{formulas for the alphas as ratios of 2-by-2 determinants} and \eqref{formulas for the alphas as diferences of 2-by-2 determinants} are well defined and \eqref{formulas for the alphas as ratios of 2-by-2 determinants} implies that $\alpha_m\neq 0$ for all $m\in\N$.

The formulas in \eqref{formulas for the alphas as ratios of 2-by-2 determinants} were originally obtained at the end of \cite[\S 7]{GeneticSumsAptekarevEtAl}.
We will see in Remark \ref{alphas are the same in Sections 3 and 4} that they have also appeared before in a completely different, combinatorial, context.

\begin{remark}
Observe that the moment matrices $\mathrm{M}^{[j]}$ of the systems $\mathrm{V}^{[j]}$ can be obtained from the moment matrix $\mathrm{M}$ of the original system $\left(v_1,\cdots,v_r\right)$, given in \eqref{moment matrix explicit}, by deleting its first $j$ columns.
Therefore, when $r=1$, the condition $\Delta^{[j]}_n>0$ for all $n\in\N$ and $j\in\{0,1\}$ reduces to the classical necessary and sufficient condition for a sequence of real numbers to be a Stieltjes moment sequence (the sequence of the moments of a positive measure on $\R^+$) \cite{StieltjesRecherchesFractionsContinues}.
Hence, we can see the condition $\Delta^{[j]}_n>0$ for all $n\in\N$ and $0\leq j\leq r$ as an extension of that condition to the context of multiple orthogonality.
\end{remark}

\begin{proof}[Proof of Corollary \ref{formulas for the alphas using determinants}]
Let $0\leq j\leq r$ and $n\geq 1$.
The Gauss-Borel factorisation $\dis\mathrm{M}^{[j]}=\mathrm{C}^{[j]}\mathrm{D}^{[j]}$ of the moment matrix of the system $V^{[j]}$ implies that
$\dis\mathrm{M}^{[j]}\left(\mathrm{D}^{[j]}\right)^{-1}=\mathrm{C}^{[j]}$
and
$\dis\left(\mathrm{C}^{[j]}\right)^{-1}\mathrm{M}^{[j]}=\mathrm{D}^{[j]}$.
%
Considering the $n$-th column and $n$-th row, respectively, of these identities, we get
\begin{align}
\mathrm{M}^{[j]}_{n+1} \, 
\begin{bmatrix} 
\left(\left(D^{[j]}\right)^{-1}\right)_{0,n} \\ \vdots \\ \left(\left(D^{[j]}\right)^{-1}\right)_{n,n} 
\end{bmatrix} 
=\begin{bmatrix} 0 \\ \vdots \\ 0 \\ 1 \end{bmatrix}
\quad\text{and}\quad
\left(\mathrm{M}^{[j]}_{n+1}\right)^\top\,
\begin{bmatrix} 
\left(\left(C^{[j]}\right)^{-1}\right)_{n,0} \\ \vdots \\  \left(\left(C^{[j]}\right)^{-1}\right)_{n,n} 
\end{bmatrix} 
=\begin{bmatrix} 0 \\ \vdots \\ 0 \\ \left(D^{[j]}\right)_{n,n} \end{bmatrix}.
\end{align}
Applying Cramer's rule, we obtain
\begin{equation}
\left(\left(D^{[j]}\right)^{-1}\right)_{n,n}=\dfrac{\Delta^{[j]}_n}{\Delta^{[j]}_{n+1}}
\quad\text{and}\quad
\left(\left(C^{[j]}\right)^{-1}\right)_{n,n-1}
=-\left(D^{[j]}\right)_{n,n}\,\dfrac{\widehat{\Delta}^{[j]}_{n+1}}{\Delta^{[j]}_{n+1}}.
\end{equation}
Taking into account that $\mathrm{D}^{[j]}$ and $\mathrm{C}^{[j]}$ are, respectively, an upper-triangular and a unit-lower-triangular matrix, we know that
$\dis\left(\left(D^{[j]}\right)^{-1}\right)_{n,n}=\left(D^{[j]}\right)_{n,n}^{-1}$
and
$\dis\left(\left(C^{[j]}\right)^{-1}\right)_{n,n-1}=-\left(C^{[j]}\right)_{n,n-1}$.
Therefore,
\begin{equation}
\left(D^{[j]}\right)_{n,n}=\dfrac{\Delta^{[j]}_{n+1}}{\Delta^{[j]}_n}
\quad\text{and}\quad
\left(C^{[j]}\right)_{n,n-1}=\dfrac{\widehat{\Delta}^{[j]}_{n+1}}{\Delta^{[j]}_n}.
\end{equation}
As a result, the identities \eqref{formulas for the alphas as ratios of 2-by-2 determinants} and \eqref{formulas for the alphas as diferences of 2-by-2 determinants} follow directly from \eqref{alpha (r+1)n} and \eqref{alpha (r+1)n+i, i from 1 to r}.

It is a direct consequence of \eqref{formulas for the alphas as ratios of 2-by-2 determinants} that $\Delta^{[j]}_n>0$ for all $n\in\N$ and $0\leq j\leq r$ implies $\alpha_m>0$ for all $m\in\N$.
Reciprocally, suppose that $\alpha_m>0$ for all $m\in\N$ and observe that \eqref{formulas for the alphas as ratios of 2-by-2 determinants} implies that, for any $n\in\N$,
\begin{align}
\Delta_{n+1}^{[r]}=\frac{\Delta_n^{[r]}\,\Delta_{n+1}^{[0]}}{\Delta_n^{[0]}}\,\alpha_{(r+1)n}
\quad\text{and}\quad
\Delta_{n+2}^{[i-1]}=\frac{\Delta_{n+1}^{[i-1]}\,\Delta_{n+1}^{[i]}}{\Delta_n^{[i]}}\,\alpha_{(r+1)n+i}
\quad\text{for }1\leq i\leq r.
\end{align}
Therefore, with the initial conditions $\dis\Delta_0^{[j]}=1$ for all $0\leq j\leq r$ and $\Delta^{[j-1]}_1=\Functional{v_j}{1}>0$ for all $1\leq j\leq r$, 
we show by induction that $\Delta^{[j]}_n>0$ for all $n\in\N$ and $0\leq j\leq r$.
\end{proof}

\section{Bidiagonal factorisations linked to branched continued fractions}
\label{Factorisations linked to BCF}
In this section, we construct bidiagonal factorisations of the form \eqref{bidiagonal matrix factorisation L_1...L_r U introduction} for the Hessenberg recurrence matrices associated with systems of multiple orthogonal polynomials via their connection with branched continued fractions.
In fact, we show that the nontrivial entries of the bidiagonal matrices are the coefficients of the corresponding branched continued fractions.

Firstly, we state the theorem that gives the foundation to our algorithm: we find the bidiagonal factorisation of the Hessenberg recurrence matrix of a system of multiple orthogonal polynomials by constructing a sequence of formal power series $\seq[m\in\N]{g_m(t)}$, determined by the moments of the orthogonality functionals (or measures) and a recurrence relation with coefficients equal to the nontrivial entries of the bidiagonal matrices in our factorisation.
We also observe that these nontrivial entries are the coefficients of a (Stieltjes-type) branched continued fraction and that they are the same as the nontrivial entries in the bidiagonal factorisations from the previous section.

Next, we give a necessary and sufficient condition on a system of multiple orthogonal polynomials for the existence of a bidiagonal factorisation of the form \eqref{bidiagonal matrix factorisation L_1...L_r U introduction} for the Hessenberg recurrence matrix of the system: the recurrence matrix of a polynomial sequence $\seq[n\in\N]{P_n(x)}$ admits such a bidiagonal factorisation if and only if it is $r$-orthogonal with respect to a system $\left(v_1,\cdots,v_r\right)$ such that the multiple orthogonal polynomials on the step-line with respect to all its Christoffel transformations $V^{[j]}$ in \eqref{systems V^[j]} exist and are unique.

Finally, we explicitly describe the sequence $\seq[m\in\N]{g_m(t)}$ used to construct our bidiagonal factorisations via the type I linear forms on the step-line with respect to the systems $V^{[j]}$ in \eqref{systems V^[j]}, establishing a so far missing link between branched continued fractions and type I multiple orthogonal polynomials.


The following theorem is the cornerstone for our method to construct bidiagonal factorisations via the connection of multiple orthogonal polynomials with branched continued fractions.
\begin{theorem}\label{Bidiagonal matrix factorisations for MOP and sequences of formal power series}
Let $\seq[n\in\N]{P_n(x)}$ be a polynomial sequence with complex coefficients and 
$\seq[m\in\N]{\alpha_m}$ be a sequence of nonzero complex numbers.
Then, the following conditions are equivalent:
\begin{enumerate}[label=(\alph*)]
\item 
$\seq[n\in\N]{P_n(x)}$ satisfies the recurrence relation 
\begin{align}
\label{recurrence relation r-OP matrix form}
x\left[P_n(x)\right]_{n\in\N}
=\mathrm{L}_1\cdots\mathrm{L}_r\,\mathrm{U}\left[P_n(x)\right]_{n\in\N},
\end{align}
involving the bidiagonal matrices $\mathrm{L}_1,\cdots,\mathrm{L}_r$, and $\mathrm{U}$ in \eqref{bidiagonal matrices U and L_k};

\item 
$\seq[n\in\N]{P_n(x)}$ is $r$-orthogonal with respect to the system of linear functionals $\left(v_1,\cdots,v_r\right)$ with moments
\begin{align}
\label{modified r-S.R. poly as moments of the orthogonality functionals}
\Functional{v_j}{x^n}=\modifiedStieltjesRogersPoly[r]{n}{j-1}{\boldsymbol{\alpha}}
\quad\text{for all }n\in\N\text{ and }1\leq j\leq r;
\end{align}

\item 
$\seq[n\in\N]{P_n(x)}$ is $r$-orthogonal with respect to a system of linear functionals $\left(v_1,\cdots,v_r\right)$ 
for which there exists a sequence of formal power series $\seq[m\in\N]{g_m(t)}$ such that
\begin{itemize}
\begin{subequations}
\item 
$g_0(t)=1$,
\hfill\refstepcounter{equation}\textup{(\theequation)}
\label{g_0(t)=1}

\item 
$\dis g_j(t)=\sum_{n=0}^{\infty}\Functional{v_j}{x^n}\,t^n$
for $1\leq j\leq r$,
\hfill\refstepcounter{equation}\textup{(\theequation)}
\label{g_j, j from 1 to r, as generating functions of the orthogonality functionals}

\item 
$\dis g_{m+1}(t)-g_m(t)=\alpha_m\,t\,g_{m+r+1}(t)$
for all $m\in\N$.
\hfill\refstepcounter{equation}\textup{(\theequation)}
\label{recurrence relation for the g_k}
\end{subequations}
\end{itemize}
The sequence $\seq[m\in\N]{g_m(t)}$ and the system $\left(v_1,\cdots,v_r\right)$ are uniquely determined by $\seq[m\in\N]{\alpha_m}$. 
\end{enumerate}
\end{theorem}
Note that, when the functionals $v_1,\cdots,v_r$ are induced by measures $\mu_1,\cdots,\mu_r$, 
\eqref{g_j, j from 1 to r, as generating functions of the orthogonality functionals} implies that the functions $\dfrac{1}{z}\,g_j\left(\dfrac{1}{z}\right)$, $1\leq j\leq r$, are the Stieltjes transforms of the measures $\mu_1,\cdots,\mu_r$, i.e.,
\begin{equation}
\label{Stieltjes transform of the measures}
\frac{1}{z}\,g_j\left(\frac{1}{z}\right)
=\sum_{n=0}^{\infty}\frac{1}{z^{n+1}}\,\int x^n\mathrm{d}\mu_j(x)
=\int\sum_{n=0}^{\infty}\frac{x^n}{z^{n+1}}\,\mathrm{d}\mu_j(x)
=\int\frac{\mathrm{d}\mu_j(x)}{z-x}
\quad\text{for all }1\leq j\leq r.
\end{equation}

To prove Theorem \ref{Bidiagonal matrix factorisations for MOP and sequences of formal power series}, we need to recall the  following method to construct Stieltjes-type branched continued fractions, introduced in \cite{AlanEtAl-LPandBCF1} as an generalisation of a method to construct classical S-fractions that goes back to the work of Euler and Gauss.
\begin{lemma}\label{Euler-Gauss method for BCF}
(cf. \cite[Prop.~2.3]{AlanEtAl-LPandBCF1})
For $r\in\Z^+$, let $\seq[m\in\N]{g_m(t)}$ be a sequence of formal power series with constant term $1$ and $f_k(t)={g_{k+1}(t)}\slash{g_k(t)}$ for all $k\in\N$.
Then, the following conditions are equivalent:
\begin{itemize}
\item 
$\seq[m\in\N]{g_m(t)}$ satisfies the recurrence relation \eqref{recurrence relation for the g_k},

\item 
$\seq[k\in\N]{f_k(t)}$ satisfies the relations in \eqref{functional equation for the generating function of r-Dyck paths}, so $f_k(t)$ is the generating function for $r$-Dyck paths at level $k$ with weights $\seq[m\in\N]{\alpha_m}$ and admits the Stieltjes-type $r$-branched-continued-fraction representation \eqref{branched cont. frac. of Stieltjes type},

\item 
the generating function of the modified $r$-Stieltjes-Rogers polynomials of type $j$ with weights $\seq[n\in\N]{\alpha_n}$ is
\begin{align}
\label{generating function of the modified m-S.-R. poly*}
\sum_{n=0}^{\infty}\modifiedStieltjesRogersPoly[r]{n}{j}{\boldsymbol{\alpha}}\,t^n
=f_0(t)\cdots f_j(t)=\frac{g_{j+1}(t)}{g_0(t)}
\quad\text{for all }j\in\N.
\end{align}
\end{itemize}
\end{lemma}

\begin{proof}[Proof of Theorem \ref{Bidiagonal matrix factorisations for MOP and sequences of formal power series}]
Accordingly to \cite[Prop.~8.2]{AlanEtAl-LPandBCF1}, the Hessenberg matrix $\mathrm{H}=\mathrm{L}_1\cdots\mathrm{L}_r\,\mathrm{U}$ is the production matrix of the matrix of generalised $r$-Stieltjes-Rogers polynomials of type $0$, $\seq[n,k\in\N]{\generalisedStieltjesRogersPolyTypeJ[r]{n}{k}{0}{\boldsymbol{\alpha}}}$.
As a result, the equivalence between \textit{(a)} and \textit{(b)} is a direct consequence of  \cite[Thm.~4.2\,\&\,Prop.~4.4]{MOPassociatedwithBCFforRatiosOfHypergeometricSeries}.

The equivalence between \textit{(b)} and \textit{(c)} follows from Lemma \ref{Euler-Gauss method for BCF}.
The sequence $\seq[m\in\N]{g_m(t)}$ and the system $\left(v_1,\cdots,v_r\right)$ in \textit{(c)} are uniquely determined by \eqref{generating function of the modified m-S.-R. poly*}, with $g_0(t)=1$, and \eqref{modified r-S.R. poly as moments of the orthogonality functionals}, respectively.
\end{proof}

As a consequence of Theorem \ref{Bidiagonal matrix factorisations for MOP and sequences of formal power series}, we can construct bidiagonal matrix factorisations for Hessenberg matrices associated with multiple orthogonal polynomials by finding sequences of formal power series $\seq[m\in\N]{g_m(t)}$ satisfying \eqref{g_0(t)=1}-\eqref{recurrence relation for the g_k} for systems of linear functionals $\left(v_1,\cdots,v_r\right)$.

Taking the derivative at $t=0$ on both sides of \eqref{recurrence relation for the g_k}, we find that if a sequence of power series $\seq[m\in\N]{g_m(t)}$ with constant term $1$ satisfies \eqref{recurrence relation for the g_k}, then
\begin{align}
\label{formula for the alphas from g_k}
\alpha_m=g_{m+1}^{\,\prime}(0)-g_m^{\,\prime}(0)
\quad\text{for all }m\in\N.
\end{align}
If $\alpha_m\neq 0$, we can invert \eqref{recurrence relation for the g_k} to obtain 
\begin{align}
\label{recurrence relation for the g_k inverted}
g_{m+r+1}(t)=\frac{1}{\alpha_m\,t}\Big(g_{m+1}(t)-g_m(t)\Big).
\end{align}

Therefore, starting from a system of linear functionals $\left(v_1,\cdots,v_r\right)$, with $\Functional{v_j}{1}=1$ for all $1\leq j\leq r$, we can construct the sequences $\seq[m\in\N]{g_m(t)}$ and $\seq[m\in\N]{\alpha_m}$ satisfying \eqref{g_0(t)=1}-\eqref{recurrence relation for the g_k} by setting $g_0(t)=1$, defining $g_1(t),\cdots,g_r(t)$ by \eqref{g_j, j from 1 to r, as generating functions of the orthogonality functionals}, and then recursively obtaining $\alpha_m$ and $g_{m+r+1}(t)$ for all $m\in\N$, 
using formulas \eqref{formula for the alphas from g_k} and \eqref{recurrence relation for the g_k inverted}, respectively, 
unless there exists $m\in\N$ such that $\alpha_m=0$.
Then, using Theorem \ref{Bidiagonal matrix factorisations for MOP and sequences of formal power series}, the nonzero coefficients $\alpha_k$ obtained correspond to nontrivial entries of the bidiagonal matrices appearing in the factorisation $\mathrm{H}=\mathrm{L}_1\cdots\mathrm{L}_r\,\mathrm{U}$ of the Hessenberg recurrence matrix of the type II multiple orthogonal polynomials on the step-line with respect to $\left(v_1,\cdots,v_r\right)$.

\begin{remark}\label{alphas are the same in Sections 3 and 4}
Let $\left(v_1,\cdots,v_r\right)$ be a system of linear functionals such that there exists a sequence of formal power series $\seq[m\in\N]{g_m(t)}$ satisfying \eqref{g_0(t)=1}-\eqref{recurrence relation for the g_k}.
Then, the nontrivial entries of the bidiagonal matrices appearing in the factorisation $\mathrm{H}=\mathrm{L}_1\cdots\mathrm{L}_r\,\mathrm{U}$ of the Hessenberg recurrence matrix for the system $\left(v_1,\cdots,v_r\right)$ are the coefficients of the branched continued fraction \eqref{branched cont. frac. of Stieltjes type} for $f_k(t)={g_{k+1}(t)}\slash{g_k(t)}$ for all $k\in\N$.

Moreover, the formulas in \eqref{formulas for the alphas as ratios of 2-by-2 determinants} were obtained in \cite[Cor.~4]{AlbenqueBouttier-Constellations},  within a completely combinatorial context where the $\seq[m\in\N]{\alpha_m}$ are the coefficients of a branched continued fraction equivalent to the one in \eqref{branched cont. frac. of Stieltjes type}.
Therefore, although the bidiagonal factorisation $\mathrm{H}=\mathrm{L}_1\cdots\mathrm{L}_r\,\mathrm{U}$ of a Hessenberg matrix is not unique when $r\geq 2$, the factorisations obtained for a system $\left(v_1,\cdots,v_r\right)$ via the methods developed in this section and in Section \ref{Factorisations using MOP coefficients} are the same.
\end{remark}

In the following result, we give a necessary and sufficient condition for the recurrence matrix of a polynomial sequence to admit a bidiagonal factorisation of the form \eqref{bidiagonal matrix factorisation L_1...L_r U introduction}.
\begin{theorem}\label{conditions for existence of bidiagonal matrix factorisation}
A polynomial sequence $\seq[n\in\N]{P_n^{[0]}(x)}$ satisfies a recurrence relation of the form
\begin{align}
\label{recurrence relation r-OP matrix form*}
x\left[P_n^{[0]}(x)\right]_{n\in\N}
=\mathrm{L}_1\cdots\mathrm{L}_r\,\mathrm{U}\left[P_n^{[0]}(x)\right]_{n\in\N},
\end{align}
involving bidiagonal matrices $\mathrm{L}_1,\cdots,\mathrm{L}_r$, and $\mathrm{U}$ of the form in \eqref{bidiagonal matrices U and L_k}, with all nontrivial entries $\seq[m\in\N]{\alpha_m}$ different from zero,
if and only if
$\seq[n\in\N]{P_n^{[0]}(x)}$ is $r$-orthogonal with respect to a system of linear functionals $V^{[0]}=\left(v_1,\cdots,v_r\right)$ such that the multiple orthogonal polynomials on the step-line with respect to all systems
\begin{align}
\label{systems V^[j], 0<=j<=r}
V^{[j]}=\left(v_{j+1},\cdots,v_r,x\,v_1,\cdots,x\,v_j\right),
\quad 0\leq j\leq r,
\end{align} 
introduced in \eqref{systems V^[j]}, exist and are unique.
\end{theorem}


To prove that the existence and uniqueness of the multiple orthogonal polynomials on the step-line with respect to all systems $V^{[j]}$, $0\leq j\leq r$, is a sufficient condition for $\seq[n\in\N]{P_n^{[0]}(x)}$ to satisfy a recurrence relation of the form \eqref{recurrence relation r-OP matrix form*},
we need to recall the definition and basic properties of $(r+1)$-fold symmetric $r$-orthogonal polynomials.

For $r\geq 1$, a polynomial sequence $\seq[n\in\N]{P_n(x)}$ is $(r+1)$-fold symmetric if
$\dis P_n\left(\e^{\frac{2\pi i}{r+1}}x\right)=\e^{\frac{2n\pi i}{r+1}}P_n(x)$ for all $n\in\N$.
This definition is equivalent to the existence of $r+1$ polynomial sequences $\dis\seq[n\in\N]{P_n^{[j]}(x)}$, $0\leq j\leq r$, 
which we refer to the as the components of the $(r+1)$-fold decomposition of $\seq[n\in\N]{P_n(x)}$,
such that
\begin{align}
\label{(r+1)-fold decomposition}
P_{(r+1)n+j}(x)=x^j\,P^{[j]}_n\left(x^{r+1}\right)
\quad\text{for all }n\in\N\text{ and }0\leq j\leq r.
\end{align}

A polynomial sequence $\seq[n\in\N]{P_n(x)}$ is $(r+1)$-fold symmetric and $r$-orthogonal if and only if
\begin{align}
\label{recurrence relation r+1 fold symmetric r-OPS}
P_{n+r+1}(x)=x\,P_{n+r}(x)-\gamma_n\,P_n(x)
\quad\text{for all }n\in\N,
\end{align}
with initial conditions $P_n(x)=x^n$ for $0\leq n\leq r$ and $\gamma_n\neq 0$ for all $n\in\N$.
When $r=1$, $(r+1)$-fold symmetric $r$-orthogonal polynomials reduce to symmetric orthogonal polynomials.
Well-known examples include the classical Hermite, Tchebyshev, Legendre, and Gegenbauer polynomials.

\begin{remark}
The Hessenberg recurrence matrix $\mathrm{H}$ of the $(r+1)$-fold symmetric $r$-orthogonal polynomial sequence satisfying the recurrence relation \eqref{recurrence relation r+1 fold symmetric r-OPS} only has nonzero entries in the supradiagonal, all equal to $1$, and in the $r^\text{th}$-subdiagonal, equal to the recurrence coefficients $\seq[n\in\N]{\gamma_n}$, i.e., 
\begin{align}
\label{Hessenberg matrix r+1-fold sym r-OP}
\mathrm{H}=
\begin{bmatrix}
	 0 & 1 \\
\vdots & \ddots & \ddots \\
	 0 & \cdots & 0 & 1 \\ 
\gamma_0 & 0 & \cdots & 0 & 1 \\
& \gamma_1 & 0 & \cdots & 0 & 1 \\
& & \ddots & \ddots & & \ddots & \ddots
\end{bmatrix}.
\end{align}
If the Hessenberg matrix in \eqref{Hessenberg matrix r+1-fold sym r-OP} admits a factorisation $\mathrm{H}=\mathrm{L}_1\cdots\mathrm{L}_r\,\mathrm{U}$ as a product of bidiagonal matrices of the form in \eqref{bidiagonal matrices U and L_k}, 
then $\alpha_0=0$ and $\alpha_0\alpha_r\cdots\alpha_{r^2}=\gamma_0\neq 0$,
which is clearly impossible.	
Therefore, $(r+1)$-fold symmetric $r$-orthogonal polynomial sequences provide an example of Hessenberg recurrence matrices which do not admit a bidiagonal matrix factorisation as in \eqref{bidiagonal matrix factorisation L_1...L_r U introduction}.
The spectral properties of Hessenberg matrices of the form in \eqref{Hessenberg matrix r+1-fold sym r-OP}, with positive and bounded coefficients $\seq[n\in\N]{\gamma_n}$, were investigated in \cite{GeneticSumsAptekarevEtAl}.
\end{remark}

\begin{proof}[Proof of Theorem \ref{conditions for existence of bidiagonal matrix factorisation}]
Suppose that 
$\seq[n\in\N]{P_n^{[0]}(x)}$ is $r$-orthogonal with respect to $V^{[0]}=\left(v_1,\cdots,v_r\right)$ such that
the multiple orthogonal polynomials on the step-line with respect to all systems
$V^{[j]}$, $0\leq j\leq r$, in \eqref{systems V^[j], 0<=j<=r} exist and are unique.
Recalling Theorem \ref{bidiagonal fatorisations using the Gauss-Borel fatorisation thm.}, we know that the Hessenberg recurrence matrices associated with the systems $V^{[j]}$ admit the bidiagonal factorisations in \eqref{eq:bidiagonal_j},
with $\alpha_m\neq 0$ for all $m\in\N$. 
In particular, $\seq[n\in\N]{P_n^{[0]}(x)}$ satisfies \eqref{recurrence relation r-OP matrix form*}.

Reciprocally, suppose that $\seq[n\in\N]{P_n^{[0]}(x)}$ satisfies \eqref{recurrence relation r-OP matrix form*}, 
with $\alpha_m\neq 0$ for all $m\in\N$.
Recalling Theorem \ref{Bidiagonal matrix factorisations for MOP and sequences of formal power series}, $\seq[n\in\N]{P_n^{[0]}(x)}$ is $r$-orthogonal with respect to a system of linear functionals $\left(v_1,\cdots,v_r\right)$ for which there exists a sequence of formal power series $\seq[m\in\N]{g_m(t)}$ satisfying \eqref{g_0(t)=1}-\eqref{recurrence relation for the g_k} with the same nonzero coefficients $\seq[m\in\N]{\alpha_m}$ appearing in \eqref{recurrence relation r-OP matrix form*}.
Let $\seq[n\in\N]{P_n(x)}$ be the $(r+1)$-fold symmetric $r$-orthogonal polynomial sequence satisfying the recurrence relation 
\begin{align}
\label{rec. rel. (r+1)-fold sym r-OPS}
P_{n+r+1}(x)=x\,P_{n+r}(x)-\alpha_n\,P_n(x)
\quad\text{for all }n\in\N,
\end{align}
with initial conditions $P_j(x)=x^j$ for $0\leq j\leq r$, 
and let $\seq[n\in\N]{P_n^{[j]}}$, $0\leq j\leq r$, be the components of the $(r+1)$-fold decomposition \eqref{(r+1)-fold decomposition} of $\seq[n\in\N]{P_n(x)}$.
Accordingly to \cite[Thm.~4.7]{BidiagonalMatrixFact.SymMOPandLP}, the sequences $\seq[n\in\N]{P_n^{[j]}}$ are $r$-orthogonal with respect to the systems $V^{[j]}$ for all $0\leq j\leq r$.
As a result, we know that the multiple orthogonal polynomials on the step-line with respect to all these systems exist and are unique.
\end{proof}

In the following result, we explicitly describe the sequence $\seq[m\in\N]{g_m(t)}$ introduced in Theorem \ref{Bidiagonal matrix factorisations for MOP and sequences of formal power series}, using the moments of the type I linear forms on the step-line with respect to the systems $V^{[j]}$ in \eqref{systems V^[j], 0<=j<=r}.
\begin{proposition}\label{description of g_k using type I linear forms}
For $r\in\Z^+$ and a system $\left(v_1,\cdots,v_r\right)$ of linear functionals, with $\Functional{v_j}{1}=1$ for all $1\leq j\leq r$, for which the multiple orthogonal polynomials on the step-line with respect to all systems $V^{[j]}$ in \eqref{systems V^[j], 0<=j<=r} exist and are unique, let $\seq[k\in\N]{q_k^{[j]}}$, $0\leq j\leq r$, be the sequences of type I linear forms on the step-line with respect to the systems $V^{[j]}$.
Then, the sequence of formal power series $\seq[m\in\N]{g_m(t)}$ 
satisfying \eqref{g_0(t)=1}-\eqref{recurrence relation for the g_k} is given by 
\begin{align}
\label{g_k as generating functions of moments of type I linear forms on the step-line}
g_0(t)=1
\quad\text{and}\quad
g_{(r+1)k+j+1}(t)=\sum_{n=0}^{\infty}\Functional{q_k^{[j]}}{x^{k+n}}\,t^n
\quad\text{for all }k\in\N\text{ and }0\leq j\leq r.
\end{align}
\end{proposition}

As mentioned in the introduction, when there exist a measure $\mu$ and weight functions $w_1(x),\cdots,w_r(x)$ defined in the support of $\mu$ such that $\dis\Functional{v_j}{f}=\int f(x)w_j(x)\mathrm{d}\mu(x)$ for any $f\in\C[x]$ and any $1\leq j\leq r$, 
we can define the type I functions $\dis Q_k^{[j]}(x)=\sum_{j=1}^{r}A_{\vec{n},j}(x)\,w_j(x)$ 
and the type I linear forms $q_k^{[j]}$ have the integral representation
$\dis\Functional{q_k^{[j]}}{f}=\int f(x)Q_k^{[j]}(x)\mathrm{d}\mu(x)$ for any $f\in\C[x]$.
In that case, 
\eqref{g_k as generating functions of moments of type I linear forms on the step-line} implies the following integral representation for the functions $g_m(t)$, $m\geq 1$, using the Stieltjes transform of the type I functions $Q_k^{[j]}(x)$, valid for all $k\in\N$ and $0\leq j\leq r$:
\begin{equation}
\label{Stieltjes transform of the type I functions}
\frac{1}{z}\,g_{(r+1)k+j+1}\left(\frac{1}{z}\right)
=\sum_{n=0}^{\infty}\frac{1}{z^{n+1}}\,\int x^{k+n}\,Q_k^{[j]}(x)\mathrm{d}\mu(x)
=\int\frac{x^k\,Q_k^{[j]}(x)\mathrm{d}\mu_j(x)}{z-x}.
\end{equation}

We can use Proposition \ref{description of g_k using type I linear forms} to obtain a combinatorial interpretation for the moments of the type I linear forms on the step-line with respect to the systems $V^{[j]}$ as generating polynomials of partial $r$-Dyck paths, establishing a link of type I multiple orthogonal polynomials with lattice paths and branched continued fractions.
Comparing \eqref{g_k as generating functions of moments of type I linear forms on the step-line} and \eqref{generating function of the modified m-S.-R. poly*}, we find that  
\begin{align}
\Functional{q_k^{[j]}}{x^{k+n}}
=\modifiedStieltjesRogersPoly[r]{n}{(r+1)k+j}{\boldsymbol{\alpha}}
=\generalisedStieltjesRogersPolyTypeJ[r]{n+k}{k}{j}{\boldsymbol{\alpha}}
\quad\text{for all }k,n\in\N\text{ and }0\leq j\leq r.
\end{align}
The second equality holds because, by definition of the modified and generalised $r$-Stieltjes-Rogers polynomials, 
both $\modifiedStieltjesRogersPoly[r]{n}{(r+1)k+j}{\boldsymbol{\alpha}}$ and $\generalisedStieltjesRogersPolyTypeJ[r]{n+k}{k}{j}{\boldsymbol{\alpha}}$
are both equal to the generating polynomial of the partial $r$-Dyck paths from $(0,0)$ to $\big((r+1)(n+k)+j,(r+1)k+j\big)$.
In addition, we have $\Functional{q_k^{[j]}}{x^m}=0=\generalisedStieltjesRogersPolyTypeJ[r]{m}{k}{j}{\boldsymbol{\alpha}}$ whenever $m<k$.
Therefore, we deduce that
\begin{align}
\label{combinatorial formulas for the moments of the type I linear forms on the step-line}
\Functional{q_k^{[j]}}{x^m}=\generalisedStieltjesRogersPolyTypeJ[r]{m}{k}{j}{\boldsymbol{\alpha}}
\quad\text{for all }k,m\in\N\text{ and }0\leq j\leq r.
\end{align}
This formula was given in \cite[Prop.~4.3]{BidiagonalMatrixFact.SymMOPandLP} without any mention of type I multiple orthogonality (in \cite{BidiagonalMatrixFact.SymMOPandLP}, the sequences $\seq[k\in\N]{q_k^{[j]}}$, $0\leq j\leq r$, appear as the dual sequences of the $r$-orthogonal polynomials with respect to the systems $V^{[j]}$).

\begin{proof}[Proof of Proposition \ref{description of g_k using type I linear forms}]
It is trivial that $\seq[m\in\N]{g_m(t)}$ defined by \eqref{g_k as generating functions of moments of type I linear forms on the step-line} satisfies \eqref{g_0(t)=1}.

By definition, $q_0^{[j]}=v_{j+1}$ for all $0\leq j\leq r-1$.
As a result, $\seq[m\in\N]{g_m(t)}$ also satisfies \eqref{g_j, j from 1 to r, as generating functions of the orthogonality functionals}.

Now we need to prove that $\seq[m\in\N]{g_m(t)}$ satisfies \eqref{recurrence relation for the g_k}.
Let $\seq[n\in\N]{P_n(x)}$ be the $(r+1)$-fold symmetric $r$-orthogonal polynomial sequence satisfying the recurrence relation \eqref{rec. rel. (r+1)-fold sym r-OPS}
and $\seq[k\in\N]{q_k}$ be the sequence of type I linear forms on the step-line dual to $\seq[n\in\N]{P_n(x)}$.
Accordingly to \cite[Prop.~4.3]{BidiagonalMatrixFact.SymMOPandLP}:
\begin{itemize}
\item 
$\Functional{q_k}{x^n}=0$ if $n-k$ is not a multiple of $r+1$, 

\item 
$\Functional{q_{(r+1)m+j}}{x^{(r+1)k+j}}=\Functional{q_m^{[j]}}{x^k}
=\generalisedStieltjesRogersPolyTypeJ[r]{m}{k}{j}{\boldsymbol{\alpha}}$ 
for all $m,k\in\N$ and $0\leq j\leq r$.
\end{itemize}
Therefore, \eqref{g_k as generating functions of moments of type I linear forms on the step-line} is equivalent to
\begin{align}
\label{g_k as generating functions of moments of type I linear forms w.r.t. (r+1)-fold sym functionals}
g_{k+1}(t)=\sum_{n=0}^{\infty}\Functional{q_k}{x^{(r+1)n+k}}\,t^n
\quad\text{for all }k\in\N.
\end{align}

The recurrence relation \eqref{rec. rel. (r+1)-fold sym r-OPS} satisfied by $\seq[n\in\N]{P_n(x)}$ can be expressed as
$\dis x\left[P_n(x)\right]_{n\in\N}=\mathrm{H}\left[P_n(x)\right]_{n\in\N}$, where $\mathrm{H}$ is the Hessenberg matrix in \eqref{Hessenberg matrix r+1-fold sym r-OP} with $\gamma_n=\alpha_n$ for all $n\in\N$.
Then, $\seq[k\in\N]{q_k}$ satisfies the dual recurrence relation 
$\dis x\left[q_k\right]_{k\in\N}=\mathrm{H}^\top\left[q_k\right]_{k\in\N}$, which can be explicitly written as
\begin{align}
x\,q_0=\alpha_0\,q_r
\quad\text{and}\quad
x\,q_k-q_{k-1}=\alpha_k\,q_{k+r}
\quad\text{for all }k\geq 1.
\end{align}

As a result, we find that
\begin{align}
g_1(t)-g_0(t)
=\sum_{n=0}^{\infty}\Functional{q_0}{x^{(r+1)(n+1)}}\,t^{n+1}
=\sum_{n=0}^{\infty}\alpha_0\,\Functional{q_r}{x^{(r+1)n+r}}\,t^{n+1}
=\alpha_0\,t\,g_{r+1}(t),
\end{align}
and
\begin{equation}
\begin{aligned}
g_{k+1}(t)-g_k(t)
&
=\sum_{n=0}^{\infty}\bigg(\Functional{q_k}{x^{(r+1)(n+1)+k+1}}-\Functional{q_{k-1}}{x^{(r+1)(n+1)+k}}\bigg)\,t^{n+1}
\\&
=\sum_{n=0}^{\infty}\alpha_k\,\Functional{q_{k+r}}{x^{(r+1)(n+1)+k}}\,t^{n+1}
=\alpha_k\,t\,g_{k+r+1}(t)
\quad\text{for all }k\geq 1.
\end{aligned}
\end{equation}
Hence, $\seq[m\in\N]{g_m(t)}$ defined by \eqref{g_k as generating functions of moments of type I linear forms on the step-line} or, equivalently, by \eqref{g_k as generating functions of moments of type I linear forms w.r.t. (r+1)-fold sym functionals} satisfies \eqref{recurrence relation for the g_k}.
\end{proof}

\section{Case study: Jacobi-Pi\~neiro polynomials}
\label{Jacobi-Pineiro section}

In this section, we present a bidiagonal factorisation of the Hessenberg recurrence matrix of the Jacobi-Pi\~neiro multiple orthogonal polynomials on the step-line,
obtained using both methods developed in Sections \ref{Factorisations using MOP coefficients} and \ref{Factorisations linked to BCF}.
At the end of the section, we give a bidiagonal matrix factorisation associated with the multiple Laguerre polynomials of first kind as a limiting case of the Jacobi-Pi\~neiro polynomials.

For a positive integer $r$ and parameters $b>-1$ and $a_1,\cdots,a_r>0$ such that $i\neq j$ implies $a_i-a_j\not\in\Z$, we consider the measures $\left(\mu_1,\cdots,\mu_r\right)$ supported on the interval $(0,1)$ with densities $\mathrm{d}\mu_j(x)=x^{a_j-1}(1-x)^b\,\mathrm{d}x$, whose moments are
\begin{align}
\label{Jacobi measures moments}
\int_{0}^{1}x^n\mathrm{d}\mu_j(x)
=\int_{0}^{1}x^{a_j+n-1}(1-x)^b\,\mathrm{d}x
=\mathrm{B}(a_j+n,b+1)
\quad\text{for all }n\in\N\text{ and }1\leq j\leq r,
\end{align}
where $\mathrm{B}\left(x,y\right)$ is the two-variable Beta function defined by
\begin{align}
\mathrm{B}(x,y)
=\int_{0}^{1}t^{x-1}(1-t)^{y-1}\mathrm{d}t
=\frac{\Gamma(x)\Gamma(y)}{\Gamma(x+y)},
\end{align}
and $\Gamma(z)$ is Euler's Gamma function.
The multiple orthogonal polynomials with respect to $\left(\mu_1,\cdots,\mu_r\right)$ are the well-known Jacobi-Pi\~neiro polynomials introduced in \cite{Pineiro}. 

Explicit expressions for the type I and type II Jacobi-Pi\~neiro polynomials with respect to an arbitrary number of weights can be found in \cite{zbMATH07872144} and \cite{zbMATH02165380}, respectively. 
These expressions involve hypergeometric series.
For $p,q\in\N$, the hypergeometric series ${}_pF_q$ (see \cite{AndrewsAskeyRoySpecialFunctions,LukeSpecialFunctionsVolI,DLMF}) is defined by
\begin{equation}
\label{generalised hypergeometric series}
\Hypergeometric[z]{p}{q}{a_1,\cdots,a_p}{b_1,\cdots,b_q}
=\sum_{n=0}^{\infty}\frac{\pochhammer{a_1}\cdots\pochhammer{a_p}} {\pochhammer{b_1}\cdots\pochhammer{b_q}}\,\frac{z^n}{n!}.
\end{equation}
where $\pochhammer{z}$ is the Pochhammer symbol defined by 
$\pochhammer[0]{z}=1$ and $\pochhammer{z}=z(z+1)\cdots(z+n-1)$ for $n\geq 1$.

When $p\geq 1$, $a_1=-n$ for some $n\in\N$, $a_2,\cdots,a_p\not\in\{0,-1,\cdots,1-n\}$ and $b_1,\cdots,b_q\not\in\{0,-1,\cdots,-n\}$, \eqref{generalised hypergeometric series} is a terminating series and defines a polynomial of degree $n$.
Otherwise, we require that $b_1,\cdots,b_q\not\in-\N$ and we treat \eqref{generalised hypergeometric series} as a formal power series, without concerning about its convergence.

\subsection{Bidiagonal factorisation linked to branched continued fractions}
\hspace*{\fill}

To construct a bidiagonal matrix factorisation of the Hessenberg recurrence matrix of the Jacobi-Pi\~neiro polynomials on the step-line via their connection with branched continued fractions, we need to find the sequence $\seq[m\in\N]{g_m(t)}$ satisfying \eqref{g_0(t)=1}-\eqref{recurrence relation for the g_k}, where $v_1,\cdots,v_r$ are the functionals with the moments in \eqref{Jacobi measures moments} normalised so that $\Functional{v_j}{1}=1$ for all $1\leq j\leq r$, i.e.,
\begin{align}
\label{Jacobi measures moments normalised}
\Functional{v_j}{x^n}
=\int_{0}^{1}\frac{x^{a_j+n-1}(1-x)^b\,\mathrm{d}x}{\mathrm{B}\left(a_j+1,b+1\right)}
=\frac{\mathrm{B}\left(a_j+n,b+1\right)}{\mathrm{B}\left(a_j,b+1\right)}
=\frac{\pochhammer{a_j}}{\pochhammer{a_j+b+1}}
\quad\text{for all }n\in\N\text{ and }1\leq j\leq r.
\end{align}
Note that, defining $v_1,\cdots,v_r$ by its moments in \eqref{Jacobi measures moments normalised}, we can replace the conditions $a_1,\cdots,a_r>0$ and $b>-1$ by the weaker condition that $a_1,\cdots,a_r,a_1+b+1,\cdots,a_r+b+1$ cannot be non-positive integers.

\begin{theorem}
\label{functions g_k for Jacobi-Pineiro poly - thm.}
For $a_1,\cdots,a_r\in\C$ such that $a_1,\cdots,a_r,a_1+b+1,\cdots,a_r+b+1$ cannot be non-positive integers, let $\seq[n\in\N]{g_n(t)}$ be the sequence defined by
\begin{align}
\label{seq. g_n of hypergeometric series for J.P. poly}
g_n(t)
=(1-t)^{b}\,
\Hypergeometric[t]{r+1}{r}{b+n_0^{\,\prime},a_1+b+n_1^{\,\prime},\cdots,a_r+b+n_r^{\,\prime}\vspace*{2 pt}} {a_1+b+n_1,\cdots,a_r+b+n_r}
\quad\text{for all }n\in\N,
\end{align}
where $\vec{n}=\left(n_1,\cdots,n_r\right)$ and $\vec{n}^{\,\prime}=\left(n_0^{\,\prime},\cdots,n_r^{\,\prime}\right)$ are the elements of the step-line of $\N^r$ and $\N^{r+1}$, respectively, such that $\left|\vec{n}\right|=\left|\vec{n}^{\,\prime}\right|=n$,
i.e.,
$\dis n_j=\ceil{\frac{{n-j+1}}{r}}$ for all $1\leq j\leq r$
and
$\dis n_i^{\,\prime}=\ceil{\frac{{n-i}}{r+1}}$ for all $0\leq i\leq r$.
Then:
\begin{itemize}[leftmargin=*]
\item 
the initial elements of $\seq[n\in\N]{g_n(t)}$ are $g_0(t)=1$ and
\begin{align}
\label{g_1 to g_r, Jacobi-Pineiro}
g_n(t)
=\sum_{k=0}^{\infty}\frac{\pochhammer[k]{a_n}}{\pochhammer[k]{a_n+b+1}}\,t^k
\quad\text{for }1\leq n\leq r.
\end{align}

\item 
$\seq[n\in\N]{g_n(t)}$ satisfies the recurrence relation 
\begin{align}
\label{rec. rel. g_k J.-P. poly}
g_{n+1}(t)-g_n(t)=\alpha_n\,t\,g_{n+r+1}(t)
\quad\text{for all }n\in\N,
\end{align}
with coefficients
\begin{align}
\label{bidiagonal matrix entries Jacobi-Pineiro via BCF}
\alpha_n
=\frac{\left(a_{j(n)}-a_{i(n)}+\floor{\dfrac{n}{r}}-\floor{\dfrac{n}{r+1}}\right) \prod\limits_{i=0}^{i(n)-1}\left(a_i+b+\floor{\dfrac{n}{r+1}}+1\right) \prod\limits_{i=i(n)+1}^{r}\left(a_i+b+\floor{\dfrac{n}{r+1}}\right)}
{\prod\limits_{j=1}^{j(n)}\left(a_j+b+\floor{\dfrac{n}{r}}+1\right) \prod\limits_{j=j(n)}^{r}\left(a_j+b+\floor{\dfrac{n}{r}}\right)},
\end{align}
for all $n\in\N$, with $a_0=0$, $i(n)=[n]_{r+1}$, and $j(n)=[n]_r+1$, where $[n]_k$ denotes the remainder of the division of $n$ by $k$.
\end{itemize}
\end{theorem}

Combining the result above with Theorem \ref{Bidiagonal matrix factorisations for MOP and sequences of formal power series}, we immediately obtain the bidiagonal factorisation of the Hessenberg recurrence matrix of the Jacobi-Pi\~neiro polynomials on the step-line as follows.
\begin{corollary}
\label{bidiagonal factorisation Jacobi-Pineiro poly}
The Hessenberg recurrence matrix of the Jacobi-Pi\~neiro polynomials on the step-line, multiple orthogonal with respect to $(\mu_1,\cdots,\mu_r)$ given by \eqref{Jacobi measures moments}, admits the factorisation $\mathrm{H}=\mathrm{L}_1\cdots\mathrm{L}_r\,\mathrm{U}$, as a product of bidiagonal matrices of the form in \eqref{bidiagonal matrices U and L_k} with nontrivial entries $\seq[n\in\N]{\alpha_n}$ given by \eqref{bidiagonal matrix entries Jacobi-Pineiro via BCF}.
\end{corollary}

The case $b=0$ is a particular case of the bidiagonal matrix factorisation for the multiple orthogonal polynomials with hypergeometric moment generating functions obtained in \cite{MOPassociatedwithBCFforRatiosOfHypergeometricSeries}.

To prove the recurrence relation \eqref{rec. rel. g_k J.-P. poly}, we need the following relation for contiguous hypergeometric series, which can be found in \cite[Lemma~14.1]{AlanEtAl-LPandBCF1}.
\begin{lemma}
\label{contiguous relation for hypergeometric series}
For $p,q\in\N$, $\vec{c}=\left(c_1,\cdots,c_p\right)\in\N^p$, and $\vec{d}=\left(d_1,\cdots,d_q\right)\in\N^q$,
\begin{align}
\Hypergeometric[t]{p}{q}{\vec{c}+\vec{e}_i}{\vec{d}+\vec{e}_j} 
-\Hypergeometric[t]{p}{q}{\vec{c}}{\vec{d}}
=\,\dfrac{\left(d_j-c_i\right)\prod\limits_{k=1,\,k\neq i}^{p}c_k}{\left(d_j+1\right)\prod\limits_{l=1}^{q}d_l}
\;t\,\Hypergeometric[t]{p}{q}{\vec{c}+1}{\vec{d}+1+\vec{e}_j},
\end{align}
where $\vec{e}_i\in\N^p$ and $\vec{e}_j\in\N^q$ are the vectors whose entries are all $0$ except the $i^{\textrm{th}}$-entry, respectively $j^{\textrm{th}}$-entry, which is equal to $1$.
\end{lemma}
This formula can be easily proved by comparing the coefficients of $t^n$ on both sides.

\begin{proof}[Proof of Theorem \ref{functions g_k for Jacobi-Pineiro poly - thm.}]
For $n=0$, 
$n_0^{\,\prime}=0$ and $n_i^{\,\prime}=n_i=0$ for all $1\leq i\leq r$.
Hence, \eqref{seq. g_n of hypergeometric series for J.P. poly} reduces to
\begin{align}
g_0(t)=(1-t)^b\Hypergeometric[t]{1}{0}{b}{-}=1.
\end{align}

For $1\leq n\leq r$, 
$n_0^{\,\prime}=1$, 
$n_i^{\,\prime}=n_i=1$ if $1\leq i\leq n-1$, 
$n_n^{\,\prime}=0$, $n_n=1$, 
and $n_i^{\,\prime}=n_i=0$ if $n+1\leq i\leq r$.
As a result, \eqref{seq. g_n of hypergeometric series for J.P. poly} reduces to
\begin{align}
\label{g_1 to g_r, Jacobi-Pineiro*}
g_n(t)
=(1-t)^b\Hypergeometric[t]{2}{1}{b+1,a_n+b}{a_n+b+1}
=\Hypergeometric[t]{2}{1}{a_n,1}{a_n+b+1}
=\sum_{k=0}^{\infty}\frac{\pochhammer[k]{a_n}}{\pochhammer[k]{a_n+b+1}}\,t^k.
\end{align}
Here, the second equality is a particular case of \cite[Eq.~15.8.1]{DLMF} 
and the final equality follows from the definition of the hypergeometric series.

Using the vectors 
$\vec{n}=\left(n_1,\cdots,n_r\right)\in\N^r$ 
and
$\vec{n}^{\,\prime}=\left(n_0^{\,\prime},\cdots,n_r^{\,\prime}\right)\in\N^{r+1}$ 
defined below \eqref{seq. g_n of hypergeometric series for J.P. poly} 
and introducing the vectors
$\vec{a}=\left(a_1,\cdots,a_r\right)\in\C^r$ 
and 
$\vec{a}^{\,\prime}=\left(0,a_1,\cdots,a_r\right)\in\C^{r+1}$, 
we can rewrite \eqref{seq. g_n of hypergeometric series for J.P. poly} as
\begin{align}
g_n(t)
=(1-t)^{b}
\Hypergeometric[t]{r+1}{r}{\vec{a}^{\,\prime}+b+\vec{n}^{\,\prime}\vspace*{2 pt}} {\vec{a}+b+\vec{n}}
\quad\text{for all }n\in\N.
\end{align}

Fix $n\in\N$.
To find the recurrence relation \eqref{rec. rel. g_k J.-P. poly}, we start by observing that
\begin{align}
\overrightarrow{n+1}=\vec{n}+\vec{e}_{j(n)}
\quad\text{and}\quad 
\overrightarrow{n+1}^{\,\prime}=\vec{n}^{\,\prime}+\vec{e}_{i(n)}^{\,\prime},
\end{align}
where 
$\vec{e}_{j(n)}=\left(e_{j(n),1},\cdots,e_{j(n),r}\right)$ 
and
$\vec{e}_{i(n)}^{\,\prime}=\left(e_{i(n),0}^{\,\prime},\cdots,e_{i(n),r}^{\,\prime}\right)$ are vectors whose entries are all equal to $0$, except $e_{i(n),i(n)}^{\,\prime}=1=e_{j(n),j(n)}$.
Therefore, we have 
\begin{align}
g_{n+1}(t)-g_n(t)
=(1-t)^{b}
\left(
\Hypergeometric[t]{r+1}{r}{\vec{a}^{\,\prime}+b+\vec{n}^{\,\prime}+\vec{e}_{i(n)}^{\,\prime}\vspace*{2 pt}} {\vec{a}+b+\vec{n}+\vec{e}_{j(n)}}
-
\Hypergeometric[t]{r+1}{r}{\vec{a}^{\,\prime}+b+\vec{n}^{\,\prime}\vspace*{2 pt}} {\vec{a}+b+\vec{n}}
\right).
\end{align}

Using Lemma \ref{contiguous relation for hypergeometric series},
\begin{equation}
\Hypergeometric[t]{r+1}{r}{\vec{a}^{\,\prime}+b+\vec{n}^{\,\prime}+\vec{e}_{i(n)}^{\,\prime}\vspace*{2 pt}} {\vec{a}+b+\vec{n}+\vec{e}_{j(n)}}
-
\Hypergeometric[t]{r+1}{r}{\vec{a}^{\,\prime}+b+\vec{n}^{\,\prime}\vspace*{2 pt}} {\vec{a}+b+\vec{n}}
=\alpha_n\,t\,\Hypergeometric[t]{r+1}{r}{\vec{a}^{\,\prime}+b+\vec{n}^{\,\prime}+1\vspace*{2 pt}} {\vec{a}+b+\vec{n}+\vec{e}_{j(n)}+1},
\end{equation}
with
\begin{align}
\label{bidiagonal matrix entries Jacobi-Pineiro via BCF*}
\alpha_n
=\dfrac{\left(a_{j(n)}-a_{i(n)}+n_{j(n)}-n_{i(n)}^{\,\prime}\right)
\prod\limits_{i=0,\,i\neq i(n)}^{r}\left(a_i+b+n_i^{\,\prime}\right)} {\left(a_{j(n)}+b+n_{j(n)}+1\right)\prod\limits_{j=1}^{r}\left(a_j+b+n_j\right)}.
\end{align}

Observe that $\vec{n}^{\,\prime}+1=\overrightarrow{n+r+1}^{\,\prime}$ and $\vec{n}+\vec{e}_{j(n)}+1=\overrightarrow{n+1}+1 =\overrightarrow{n+r+1}$.
Hence,
\begin{align}
g_{n+1}(t)-g_n(t)
=\alpha_n\,t\,(1-t)^{b}\,
\Hypergeometric[t]{r+1}{r}{\vec{a}^{\,\prime}+b+\overrightarrow{n+r+1}^{\,\prime}\vspace*{2 pt}} {\vec{a}+b+\overrightarrow{n+r+1}}
=\alpha_n\,t\,g_{n+r+1}(t).
\end{align}

Finally, note that 
\begin{align}
n_j=
\begin{cases}
	\floor{\frac{{n}}{r}}+1 & \text{if } j\leq j(n)-1, \vspace*{2 pt} \\
\hfil \floor{\frac{{n}}{r}} & \text{if } j\geq j(n),
\end{cases}
\quad\text{and}\quad
n_i^{\,\prime}=
\begin{cases}
	\floor{\frac{{n}}{r+1}}+1 & \text{if } i\leq i(n)-1, \vspace*{2 pt} \\
\hfil \floor{\frac{{n}}{r+1}} & \text{if } i\geq i(n).
\end{cases}
\end{align}
Therefore, \eqref{bidiagonal matrix entries Jacobi-Pineiro via BCF*} is equivalent to \eqref{bidiagonal matrix entries Jacobi-Pineiro via BCF}, concluding our proof. 
\end{proof}

\subsection{Bidiagonal factorisation using the explicit expression for the type I polynomials}
\hspace*{\fill}

Here, we find again the bidiagonal factorisation described in Corollary \ref{bidiagonal factorisation Jacobi-Pineiro poly} for the Hessenberg recurrence matrix of the Jacobi-Pi\~neiro polynomials, now using explicit expressions for the type I polynomials recently given in \cite{zbMATH07872144}.
We consider the system of measures $\left(\mu_1,\cdots,\mu_r\right)$ supported on the interval $(0,1)$ with densities $\mathrm{d}\mu_j(x)=x^{a_j-1}(1-x)^b\,\mathrm{d}x$, with $b>-1$ and $a_1,\cdots,a_r>0$ such that $i\neq j$ implies $a_i-a_j\not\in\Z$.
Then, for any $1\leq j\leq r$, we have:
\begin{itemize}
\item
$\dis \mu_j\left(a_1,\cdots,a_r;b\right)=\mu_1\left(a_j,\cdots,a_r,a_1+1,\cdots,a_{j-1}+1;b\right)$,

\item
$\dis x\,\mu_j\left(a_1,\cdots,a_r;b\right)=\mu_j\left(a_1+1,\cdots,a_r+1;b\right)=
\mu_1\left(a_j+1,\cdots,a_r+1,a_1+2,\cdots,a_{j-1}+2;b\right)$.
\end{itemize}
Therefore, the multiple orthogonal polynomials with respect to the system 
$\dis\left(\mu_{i+1},\cdots,\mu_r,x\,\mu_1,\cdots,x\,\mu_i\right)$, $0\leq i\leq r$,
are the Jacobi-Pi\~neiro polynomials with parameters $a_i,\cdots,a_r,a_1+1,\cdots,a_{i-1}+1$, and $b$.

As a result, it follows from \cite[Thm.~2]{zbMATH07872144} that, in this case, the coefficients of the type I multiple orthogonal polynomials appearing in \eqref{bidiagonal matrix entries using type I MOP 1}-\eqref{bidiagonal matrix entries using type I MOP 2} are, for $m\in\N$ and $0\leq k\leq r-1$,
\begin{align}
a^{[i]}_{rm+k,k+1}(m)=
\frac{\Gamma\left(a_{k+i}+b+(r+1)m+k\right)\prod\limits_{j=1}^{k}\left(a_{j+i}+b+rm+k\right)_{m+1} \prod\limits_{j=k+1}^{r}\left(a_{j+i}+b+rm+k\right)_{m}} {m!\,\Gamma\left(a_{k+i}+m\right)\Gamma(b+rm+k)\prod\limits_{j=1}^{k-1}\left(a_{k+i+1}-a_{j+i}\right)_{m+1} \prod\limits_{q=k+1}^{r}\left(a_{k+i+1}-a_{j+i}+1\right)_{m}},
\end{align}
where $a_{rn+j}=a_j+n$ for all $0\leq j\leq r$ and $n\in\N$.

Inputting these coefficients in \eqref{bidiagonal matrix entries using type I MOP 1}-\eqref{bidiagonal matrix entries using type I MOP 2}, we derive the formulas
\begin{align}
\label{JPBidiagonal-i=0}
\alpha_{(r+1)(rm+k)}
=\frac{(a_{k+1}+m)\prod\limits_{j=1}^{r}\left(a_j+b+rm+k\right)} {\prod\limits_{j=1}^{r+1}\left(a_{k+j}+b+(r+1)m+k\right)},
\end{align}
and
\begin{align}
\label{JPBidiagonal-i>=1}
\alpha_{(r+1)(rm+k)+i}
=\frac{(a_{k+i+1}-a_{i}+m)(b+rm+k+1)\prod\limits_{j=1}^{r-1}\left(a_{i+j}+b+rm+k\right)} {\prod\limits_{j=1}^{r+1}\left(a_{k+i+j}+b+(r+1)m+k\right)},
\end{align}
for the nontrivial entries of the bidiagonal matrices appearing in the factorisation $\mathrm{H}=\mathrm{L}_1\cdots\mathrm{L}_r\,\mathrm{U}$ of the Hessenberg recurrence matrix of the Jacobi-Pi\~neiro polynomials on the step-line.
It is a straightforward computation to check that \eqref{JPBidiagonal-i=0}-\eqref{JPBidiagonal-i>=1} is equivalent to \eqref{bidiagonal matrix entries Jacobi-Pineiro via BCF}.

When $r=2$, the coefficients in \eqref{bidiagonal matrix entries Jacobi-Pineiro via BCF} and \eqref{JPBidiagonal-i=0}-\eqref{JPBidiagonal-i>=1} coincide with the ones found in \cite[\S~8.1]{GeneticSumsAptekarevEtAl}.

It is a direct consequence of \eqref{bidiagonal matrix entries Jacobi-Pineiro via BCF} or \eqref{JPBidiagonal-i=0}-\eqref{JPBidiagonal-i>=1} that $\dis\alpha_n\to\frac{r^r}{(r+1)^{r+1}}$ as $n\to\infty$.
Therefore, recalling \eqref{recurrence coefficients as a combination of BCF coefficients}, the coefficients of the recurrence relation \eqref{recurrence relation r-OP} satisfied by the type II Jacobi-Pi\~neiro polynomials on the step-line have the same limit as in \cite[Eq.~8.17]{MOPassociatedwithBCFforRatiosOfHypergeometricSeries}:
$\dis\gamma_n^{[k]}\to\binom{r+1}{k+1}\left(\frac{r^r}{(r+1)^{r+1}}\right)^{k+1}$ as $n\to\infty$.

\subsection{Multiple Laguerre of first kind as a limiting case}
\hspace*{\fill}

For a positive integer $r$ and parameters $a_1,\cdots,a_r>0$ such that $a_i-a_j\not\in\Z$ whenever $i\neq j$, we consider the measures $\left(\nu_1,\cdots,\nu_r\right)$ supported on the interval $(0,+\infty)$ with densities $\mathrm{d}\nu_j(x)=\e^{-x}\,x^{a_j-1}\mathrm{d}x$ for all $1\leq j\leq r$, whose moments are
\begin{align}
\label{Laguerre measures moments}
\int_{0}^{+\infty}x^n\mathrm{d}\nu_j(x)
=\int_{0}^{+\infty}\e^{-x}\,x^{a_j+n-1}\,\mathrm{d}x
=\Gamma(a_j+n)
\quad\text{for all }n\in\N\text{ and }1\leq j\leq r.
\end{align}

The limiting relation between the Jacobi and Laguerre measures is 
\begin{align}
\mathrm{d}\nu_j(x)=\lim\limits_{b\to\infty}\mathrm{d}\mu_j\left(\dfrac{x}{b}\right)
\quad\text{for all }1\leq j\leq r,
\end{align} 
where $\mu_1,\cdots,\mu_r$ are the measures determined by \eqref{Jacobi measures moments}.
Therefore, the multiple Laguerre polynomials of the first kind can be obtained as a limiting case of Jacobi-Pi\~neiro polynomials.
In particular, if $\seq[n\in\N]{P_n(x)}$ and $\seq[n\in\N]{\hat{P}_n(x)}$ are the sequences of type II multiple orthogonal polynomials on the step-line with respect to $(\mu_1,\cdots,\mu_r)$ given by \eqref{Jacobi measures moments} and to $(\nu_1,\cdots,\nu_r)$ given by \eqref{Laguerre measures moments}, respectively, then
\begin{equation}
\hat{P}_n(x)=\lim_{b\to\infty}b^n\,P_n\left(\frac{x}{b}\right)
\quad\text{for all }n\in\N.
\end{equation}
As a consequence, the recurrence coefficients $\dis\left(\gamma_n^{[k]}\right)_{n\in\N}^{0\leq k\leq r}$ and  $\dis\left(\hat{\gamma}_n^{[k]}\right)_{n\in\N}^{0\leq k\leq r}$ for the type II Jacobi-Pi\~neiro polynomials and type II multiple Laguerre polynomials of the first kind on the step-line are related by
\begin{equation}
\hat{\gamma}_n^{[k]}=\lim_{b\to\infty}b^{k+1}\,\gamma_n^{[k]}
\quad\text{for all }n\in\N\text{ and }0\leq k\leq r.
\end{equation}
Hence, recalling \eqref{recurrence coefficients as a combination of BCF coefficients}, we obtain the following result.
\begin{theorem}
The Hessenberg recurrence matrix of the multiple Laguerre polynomials of the first kind on the step-line, multiple orthogonal with respect to $(\nu_1,\cdots,\nu_r)$ given by \eqref{Laguerre measures moments}, admits the factorisation $\mathrm{H}=\mathrm{L}_1\cdots\mathrm{L}_r\,\mathrm{U}$, as a product of the bidiagonal matrices in \eqref{bidiagonal matrices U and L_k} with nontrivial entries $\seq[n\in\N]{\hat{\alpha}_n}$ such that
$\hat{\alpha}_n=\lim\limits_{b\to\infty}b\,\alpha_n$ for all $n\in\N$,
where $\seq[n\in\N]{\alpha_n}$ are the nontrivial entries in the bidiagonal decomposition of the recurrence matrix of the Jacobi-Pi\~neiro polynomials given in \eqref{bidiagonal matrix entries Jacobi-Pineiro via BCF} and \eqref{JPBidiagonal-i=0}-\eqref{JPBidiagonal-i>=1}.
Therefore, 
\begin{align}
\label{multiple Laguerre alphas}
\hat{\alpha}_n=a_{j(n)}-a_{i(n)}+\floor{\dfrac{n}{r}}-\floor{\dfrac{n}{r+1}}
\quad\text{for all }n\in\N,
\end{align}
with $a_0=0$, $i(n)=[n]_{r+1}$, and $j(n)=[n]_r+1$, where $[n]_k$ denotes the remainder of the division of $n$ by $k$, 
or, alternatively,
\begin{align}
\label{multiple Laguerre alphas*}
&
\hat{\alpha}_{(r+1)(rm+k)+i}
=\begin{cases}
\hfil a_{k+i+1}-a_i+m & \text{if }k+i\leq r, \\
	a_{k+i+1-r}-a_i+m & \text{if }k+i\geq r+1,
\end{cases}
\end{align}
for all $m\in\N$, $0\leq k\leq r-1$, and $0\leq i\leq r$.
\end{theorem}

It is a direct consequence of either \eqref{multiple Laguerre alphas} or \eqref{multiple Laguerre alphas*} that $\dis\hat{\alpha}_n\sim\frac{n}{r(r+1)}$ as $n\to\infty$.
Therefore, recalling \eqref{recurrence coefficients as a combination of BCF coefficients}, 
the recurrence coefficients $\dis\left(\hat{\gamma}_n^{[k]}\right)_{n\in\N}^{0\leq k\leq r}$ for the type II multiple Laguerre polynomials of the first kind on the step-line satisfy the asymptotic behaviour 
$\dis\hat{\gamma}_n^{[k]}\sim\binom{r+1}{k+1}\left(\frac{n}{r}\right)^{k+1}$ as $n\to\infty$.

Note that, alternatively to taking the limit in \eqref{multiple Laguerre alphas}, we could apply the algorithms developed in Sections \ref{Factorisations using MOP coefficients} and \ref{Factorisations linked to BCF} directly to the multiple Laguerre polynomials of first kind.
For the method from Section \ref{Factorisations using MOP coefficients}, we could input the formulas for the type I multiple Laguerre polynomials of the first kind with respect to an arbitrary number of weights obtained in \cite[Thm.~3]{zbMATH07872144} in \eqref{bidiagonal matrix entries using type I MOP 1}-\eqref{bidiagonal matrix entries using type I MOP 2} to obtain the coefficients $\seq[n\in\N]{\hat{\alpha}_n}$.
However, it is much easier to simply take the limit in \eqref{multiple Laguerre alphas}. 

Using the algorithm from Section \ref{Factorisations linked to BCF}, 
the nontrivial entries $\seq[n\in\N]{\hat{\alpha}_n}$ 
of the bidiagonal decomposition of the recurrence matrix of the multiple Laguerre polynomials of first kind 
are the coefficients of the recurrence relation
\begin{align}
\label{recurrence relation for the g_k multiple Laguerre}
\hat{g}_{k+1}(t)-\hat{g}_k(t)=\hat{\alpha}_k\,t\,\hat{g}_{k+r+1}(t)
\quad\text{for all }k\in\N,
\end{align}
with initial conditions $\hat{g}_0(t)=1$ and
\begin{align}
\label{g_1 to g_r, Laguerre}
\hat{g}_k(t)
=\sum_{n=0}^{\infty}\frac{\Gamma(a_k+n+1)}{\Gamma(a_k+1)}\,t^n
=\sum_{n=0}^{\infty}\pochhammer{a_k+1}\,t^n
=\Hypergeometric[t]{2}{0}{a_k+1,1}{-}
\quad\text{for }1\leq k\leq r.
\end{align}
This relation is satisfied by the sequence $\seq[n\in\N]{\hat{g}_n(t)}$ obtained by taking the limit
$\hat{g}_n(t)=\lim\limits_{b\to\infty}g_n(bt)$ for all $n\in\N$, 
where $\seq[k\in\N]{g_n(t)}$ is the sequence defined by \eqref{seq. g_n of hypergeometric series for J.P. poly}.
The computation of this limit is connected to contour integral representations for type I Jacobi-Pi\~neiro and multiple Laguerre polynomials and will be considered in forthcoming work.

\section*{Acknowledgements}
\textbf{AB} acknowledges  financial support by the Centre for Mathematics of the University of Coimbra (CMUC, https://doi.org/10.54499/UID/00324/2025) under the Portuguese Foundation for Science and Technology (FCT), Grants UID/00324/2025 and UID/PRR/00324/2025.
\textbf{JEFD}, \textbf{AF}, and \textbf{HL} acknowledge Center for Research and Development in Mathematics and Applications (CIDMA) from University of Aveiro funded by the Portuguese Foundation for Science and Technology (FCT) through grants UID/4106/2025 and UID/PRR/4106/2025.
Additionally, \textbf{JEFD} acknowledges PhD contract DOI: 10.54499/UI/BD/152576/2022 from FCT.
\textbf{JEFD} and \textbf{MM} acknowledge research project [PID2021- 122154NB-I00], Ortogonalidad y Aproximaci\'on con Aplicaciones en Machine Learning y Teor\'ia de la Probabilidad funded by MICIU/AEI/ 10.13039/501100011033 and by ``ERDF A Way of making Europe''.

\bibliographystyle{plain}

\end{document}